\documentclass[preprint]{imsart}

\RequirePackage[OT1]{fontenc}
\RequirePackage{amsthm,amsmath}
\RequirePackage[numbers]{natbib}
\RequirePackage[colorlinks,citecolor=blue,urlcolor=blue]{hyperref}

\usepackage{verbatim}
\usepackage{bbm}
\usepackage{graphicx} 
\usepackage[cmex10]{amsmath}
\usepackage{amssymb}
\usepackage{mathrsfs}
\usepackage{verbatim}
\usepackage{epstopdf}

\startlocaldefs
\numberwithin{equation}{section}
\theoremstyle{plain}
\newtheorem{theorem}{Theorem}[section]

\newtheorem{lemma}{Lemma}[section]

\newtheorem{corollary}{Corollary}[section]
\endlocaldefs

\newcommand{\sd}{{\sf sd}}
\newcommand{\Var}{{\sf Var}}
\newcommand{\Exp}{{\sf E}}
\newcommand{\Cov}{{\sf Cov}}

\newcommand{\Pro}{{\sf P}}

\newcommand{\D}{\Delta}
\newcommand{\oD}{\overline{\Delta}}
\newcommand{\uD}{\underline{\Delta}}
\newcommand{\oG}{\overline{\Gamma}}
\newcommand{\uG}{\underline{\Gamma}}
\newcommand{\G}{\Gamma}
\newcommand{\Gi}{\Gamma^{k}}
\newcommand{\Di}{\Delta^{k}}

\newcommand{\cD}{\mathcal{D}}

\newcommand{\cN}{\mathcal{N}}

\newcommand{\cL}{\mathcal{L}}

\newcommand{\calo}{\mathcal{O}}

\begin{document}

\begin{frontmatter}
\title{Low-rate renewal theory and estimation} 
\runtitle{Low-rate renewal theory and estimation}

\begin{aug}
\author{\fnms{Georgios} \snm{Fellouris}\ead[label=e1]{fellouri@illinois.edu}}

\runauthor{G. Fellouris}

\affiliation{University of Illinois, Urbana-Champaign}

\address{University of Illinois, Urbana-Champaign \\
Department of Statistics \\
725 South Wright Street, \\
Illini Hall, 116 
Champaign, IL 61820 USA\\
\printead{e1}\\
\phantom{E-mail: fellouri@illinois.edu}} 
\end{aug}

\begin{abstract}
Certain renewal theorems are extended to the case that the rate of the renewal process goes to 0 and, more generally, to the  case that the drift of the random walk goes to infinity.  These extensions are motivated by and applied to the problem of decentralized parameter estimation under severe communication constraints. 
\end{abstract}


\end{frontmatter}

\section{Introduction}
\subsection{Low-rate renewal theory} \label{sub11}
Let $(\tau_{n})_{n \in \mathbb{N}}$ be a random walk with finite, positive drift $\delta$. That is, $(\tau_{n}-\tau_{n-1})_{n \in \mathbb{N}}$ are
independent and identically distributed  (i.i.d.) random variables with mean $\delta:=\Exp[\tau] \in (0, \infty)$, where $\tau:=\tau_{1}$ and  $\tau_{0}:=0$. When $\Pro(\tau>0)=1$, $(\tau_{n})$ can be thought of as a sequence of random times at which certain events occur. Then, $N(t):=\max\{n \in \mathbb{N}: \tau_{n} \leq t\}$ counts the number of events that have occurred up to time $t$ and the ``empirical rate'', $N(t)/t$, converges in $\cL^{r}$  to the ``theoretical rate'', $1/\delta$, as $t \rightarrow \infty$, i.e., 
\begin{equation} \label{rate1}
\frac{N(t)}{t} \Big/ \frac{1}{\delta} \overset{\cL^{r}}  \longrightarrow  1 \quad \text{as} \quad t \rightarrow \infty
\end{equation}
for any $r>0$.  
When $\Pro(\tau> 0)< 1$, (\ref{rate1}) is no  longer valid, but it can be recovered if we replace $N(t)$ by the first-passage time
$\nu(t):=\inf\{n \in \mathbb{N}: \tau_{n} > t\}$. Specifically, if $\Exp[(\tau^{-})^{r}]<\infty$ for some $r\geq 1$, then 
\begin{equation} \label{rate2}
\frac{\nu(t)}{t} \Big/ \frac{1}{\delta} \overset{\cL^{r}}  \longrightarrow 1  \quad \text{as} \quad t \rightarrow \infty.
\end{equation}
Here, and in what follows,  $\tau^{+}:=\max\{\tau,0\}$ and $\tau^{-}:=\max\{-\tau,0\}$ represent the positive and negative part of $\tau$, respectively. 

These asymptotic results are well known and it would be fair to say that they constitute the cornerstones of renewal theory;  (\ref{rate1})  goes back to Doob \cite{doob} for $r=1$  and  Hatori \cite{hatori} for  $r>1$;  
 (\ref{rate2}) is due to Chow and Robbins \cite{chowrob} for $r=1$, Chow \cite{chow} for $r=2$ and Gut \cite{gut} for arbitrary $r \geq 1$. For more details on renewal theory we refer to Asmussen \cite{asmu} and Gut \cite{gut}. 
 
Our first contribution in the present work is that we provide sufficient conditions on the moments of  $\tau$ so that these classical, asymptotic results 
remain valid when $\delta$ is not fixed,  but $\delta \rightarrow \infty$ so that $\delta=o(t)$. Specifically, in this asymptotic setup  we show that  (\ref{rate2}) is preserved when   
\begin{equation} \label{condi1}
\Exp[(\tau^{+})^{r+1}]=\calo(\delta^{r+1}) \quad \text{and} \quad \Exp[|\tau-\delta|^{r \vee 2}]=\calo(\delta^{r \vee 2}),
\end{equation} 
where $r \vee 2:=\max\{r, 2\}$. 
When, in particular, $\Pro(\tau>0)=1$, this condition reduces to  $\Exp[\tau^{r+1}]=\calo(\delta^{r+1})$ and it 
guarantees that (\ref{rate1}) is preserved as $\delta \rightarrow \infty$ so that $\delta=o(t)$. 
 
In addition to these extensions, we also establish a version of Anscombe's theorem \cite{ans} in the same, ``low-rate'' spirit. 
Thus, let $(Y_{t})_{t \in \mathbb{N}}$ be an arbitrary  sequence of random variables  that converges in distribution as $t \rightarrow \infty$ to some random variable $Y_{\infty}$, i.e., 
$Y_{t} \overset{D} \longrightarrow   Y_{\infty}$ as $t \rightarrow \infty$. Suppose that $(Y_{t})_{t \in \mathbb{N}}$ is observed only at a sequence, $(\tau_{n})_{n \in \mathbb{N}}$, of random times  that form a renewal process with rate  $1/\delta$. If $Y$ is uniformly continuous in probability,  
from classical renewal theory and  Anscombe's theorem \cite{ans} it follows that convergence is preserved when $t$ is replaced by the last sampling time, $\tau(t):=\tau_{N(t)}$, i.e., 
\begin{equation}  \label{law} 
 Y_{\tau(t)} \overset{D} \longrightarrow   Y_{\infty}  \quad \text{as} \quad t \rightarrow \infty
 \end{equation}
for any fixed $\delta$. Is this also the case when $\delta \rightarrow \infty$ so that $\delta=o(t)$? We give a positive answer to this question under the assumption that $\Var[\tau]=\calo(\delta^{2})$. 

The proofs of the above results are based on properties of stopped random walks, for which we refer to Gut \cite{gut}, and  Lorden's \cite{lorden} inequalities on the excess over a boundary and the age of a renewal process. Moreover, they provide alternative proofs of the classical, fixed-rate results under additional moment assumptions on $\tau$.

\subsection{Motivation} 
When $\Pro(\tau>0)<1$, the asymptotic setup ``$\delta,t \rightarrow \infty$ so that $\delta=o(t)$'' implies that the random walk $(\tau_{n})$ has ``large'' drift $\delta$ and 
crosses a positive level $t$ that is much larger than $\delta$. When $\Pro(\tau>0)=1$, it implies that the renewal process $(\tau_{n})$ is observed in an interval $[0,t]$ and the average time, $\delta$,  between two consecutive events  is ``large'', yet much smaller than the length of observation window, $t$.  

We are not aware of similar low-rate/large-drift  extensions of (\ref{rate1})-(\ref{rate2}) and (\ref{law}). On the contrary, extensions and ramifications of the above results in the literature often require random walks with \textit{small drifts} and renewal processes with \textit{high rates}. For example, when $(\tau_{n})$ is a random walk and the distribution of $\tau$ can be embedded in an exponential family, Siegmund \cite{sieg2} computed the limiting values of the moments of the overshoot $\tau_{\nu(t)}-t$ as $\delta \rightarrow 0$ and $t \rightarrow \infty$ so that $\delta t=\calo(1)$. For similar ``corrected diffusion approximations'' that require random walks with small drifts we refer to  Chang\cite{cha}, Chang and Peres \cite{peres},  Blanchet and Glynn  \cite{bla}. 

When $(\tau_n)$ is a renewal process that corresponds to the sampling times of some stochastic process, as for example in Mykland and A\"it-Sahalia \cite{myk} or  Rosenbaum and Tankov \cite{rose}, limit theorems are also obtained as $\delta \rightarrow 0$, that is as the sampling rate goes to infinity. 
While such a \textit{high-frequency}  setup is natural in certain applications, such as mathematical finance (see, e.g., Zhang et. al \cite{zhang}, Florescu et. al \cite{flo}), and commonplace in the statistical inference of stochastic processes (see, e.g., A\"it-Sahalia  and Jacod \cite{aij2}, 
Tudor and Viens \cite{viens}), in other application areas it is often desirable to have \textit{low-frequency} sampling.

This is, for example, the case when the process of interest is being monitored at some remote location by a network of (low-cost, battery-operated) sensors,  which transmit their data to a fusion center subject to bandwidth and energy constraints \cite{for}. This is often the situation in environmental monitoring, intelligent transportation, space exploration. In such applications, infrequent communication  leads to significant energy gains, since it implies that for long time-periods the sensors can only sense the environment, without transmitting data  to the fusion center \cite{park}. On the other hand, limited bandwidth requires that each sensor transmit a \textit{low-bit} message  whenever it communicates with the fusion center \cite{tsi}.
Here, we consider a parameter estimation problem under such communication constraints, with the goal of demonstrating the usefulness of the aforementioned extensions of renewal theory.

\subsection{Decentralized parameter estimation} \label{sub13}
Consider $K$ dispersed sensors that take repeated measurements of  some quantity $\mu$ subject to measurement error. Thus, we assume that 
each sensor $k$ takes a sequence of i.i.d. observations $(X_{t}^{k})_{t \in \mathbb{N}}$ with  expectation $\Exp[X_{1}^{k}]=\mu$,
standard deviation $\sd[X^{k}_{1}]= \sigma_{k}$ and that observations are independent across sensors. 
The sensors transmit their data to a fusion center, which is responsible for estimating $\mu$. However, due to bandwidth constraints, they are allowed to 
transmit only low-bit messages.  The problem is to decide $(i)$ what messages each sensor transmits to the fusion center and  $(ii)$ how the fusion center 
utilizes these messages in order to estimate $\mu$.

This problem has been considered extensively in the engineering literature under the assumption that $\sigma_{1}, \ldots, \sigma_{K}$ are known and often assuming that $X_{1}^{k}, 1 \leq k \leq K$ are identically distributed  and/or that their densities have bounded support (see, e.g., Ribeiro and Giannakis \cite{ribe2},
Xiao and Luo \cite{luo3}, Msechu and Giannakis \cite{gia} and the references therein). However, the proposed estimators in all these papers do \textit{not} attain the asymptotic distribution of the optimal (in a mean-square sense) linear estimator that would be employed by the fusion center in the case that it had full access to the sensor observations. Indeed, if each sensor $k$ transmitted its exact observation $X_{t}^{k}$ at any time $t$, the fusion center could estimate $\mu$ at time $t$ with 
\begin{equation} \label{opti}
\mu_{t}:=\sum_{k=1}^{K} w_{k} \mu_{t}^{k}, \quad \mu_{t}^{k}:= S_{t}^{k}/t, \quad S_{t}^{k}:=X_{1}^{k}+\ldots+X_{t}^{k},
\end{equation}
where $w_{k}$ is proportional to $(\sigma_{k})^{-2}$ and $\sum_{k=1}^{K} w_{k}=1$, and from the Central Limit Theorem we would  have 
\begin{equation} \label{CLT}
\frac{\mu_{t} -\mu}{\sqrt{K} \sigma /\sqrt{t}} \rightarrow \cN(0,1) \quad \text{as} \quad t \rightarrow \infty,
\end{equation}
where $\sigma^{2}:= \sum_{k=1}^{K} \sigma_{k}^{2} /K$. It has recently been shown that it is possible to obtain an estimator that
attains this asymptotic distribution, even though it requires minimal transmission activity.  Indeed, suppose that each sensor $k$ communicates with the fusion center whenever $(S_{t}^{k})_{t \in \mathbb{N}}$ has changed by  $\Di>0$ since the previous communication time. Then, the sequence of communication times of sensor $k$ is $\tau_{n}^{k}:= \inf  \{t \geq \tau_{n-1}^{k}: |S^{k}_{t}-S^{k}_{\tau_{n-1}}| \geq \Di\}$ and $N^{k}(t):=\max\{n \in \mathbb{N}: \tau_{n}^{k} \leq t\}$ is the number of its transmitted messages up to time $t$. At time $\tau_{n}^{k}$, sensor $k$ transmits (with only 1 bit) the value of the indicator
$z_{n}^{k}:=1(S^{k}_{\tau_{n}^{k}}-S^{k}_{\tau_{n-1}} \geq \Di)$, informing the fusion center whether the upper or the lower threshold was crossed, but not about the size of the overshoot, $\eta_{n}^{k}:=|S^{k}_{\tau_{n}^{k}}-S^{k}_{\tau_{n-1}}|-\Di$.  Based on this information, the fusion center can estimate $\mu$ at any time $t$ with 
\begin{equation} \label{old}
\check{\mu}_{t}:=\sum_{k=1}^{K} w_{k} \check{\mu}_{t}^{k}, \quad  \check{\mu}_{t}^{k}:= \frac{\Di}{t} \, \sum_{n=1}^{N_{t}^{k}} (2 z_{n}^{k}-1).
\end{equation}
It can be shown, as in Fellouris \cite{felaos} (Section 4.2), that $\check{\mu}_{t}$ preserves (\ref{CLT}) when each $(X_{t}^{k})_{t \in  \mathbb{N}}$ results from \textit{high-frequency} sampling of an underlying Brownian motion. More generally, if each sensor $k$ can transmit $d_{k}\geq1$ bits whenever it communicates with the fusion center, then at each time $\tau_{n}^{k}$ it can transmit information not only about the sign, but also about the size of the overshoot, $\eta_{n}^{k}$. In this case, the resulting estimator will preserve (\ref{CLT}) as long as the number of bits per transmission, $d_{k}$, goes to infinity as $t \rightarrow \infty$ for every  $1 \leq k  \leq K$ (see   Yilmaz and Wang \cite{yil} for the case of Gaussian observations).

The above results, however, do not answer the question whether it is possible to construct an estimator at the fusion center that requires transmission of  \textit{1-bit} messages from the sensors and, nevertheless, achieves the asymptotic distribution (\ref{CLT}) for \textit{any}  distribution the sensor observations, $X_{1}^{k}$, $1 \leq k \leq K$, may follow.  In this article, we show that this is indeed possible when sensor $k$ transmits the 1-bit message $z_{n}^{k}$ at each time $\tau_{n}^{k}$, as long as the fusion center estimator is modified as follows:
\begin{equation} \label{new}
\hat{\mu}_{t}:=\sum_{k=1}^{K} w_{k} \hat{\mu}_{t}^{k}, \quad  \hat{\mu}_{t}^{k}:= \frac{\Di}{\tau^{k}(t)} \, \sum_{n=1}^{N_{t}^{k}} (2 z_{n}^{k}-1)
\end{equation}
where $\tau^{k}(t):=\tau^{k}_{N^{k}(t)}$ is the last communication time from sensor $k$. More specifically, we show that (\ref{CLT}) remains valid 
when $\mu_{t}$  is replaced by $\hat{\mu}_{t}$, as long as $\Di \rightarrow \infty$ so that $\sqrt{t} << \Di <<t$ for every $1 \leq k \leq K$. In other words, the communication rate from the sensors to the fusion center  \textit{ must be sufficiently low} for $\hat{\mu}_{t}$ to be efficient, not only from a practical, but also from a statistical point of view. 

Note also that $\hat{\mu}_{t}$ does not require knowledge of the  distribution of $X_{1}^{k}$, $1 \leq k \leq K$. When, however, these distributions are known and satisfy certain integrability conditions, we can construct a modification of $\hat{\mu}_{t}$ that preserves asymptotic distribution (\ref{CLT}) when $\sqrt[4]{t} << \Di <<t$ for every $1 \leq k  \leq K$ (that is, for even higher rates of communication), whereas we can also construct a modification of $\check{\mu}_{t}$ that, contrary to $\check{\mu}_{t}$, does  attain (\ref{CLT}), as long as $\sqrt[4]{t} << \Di <<\sqrt{t}$ for every $1 \leq k \leq K$.  Finally,  we consider the estimation of $\sigma$, when $\sigma_{k}$, $1 \leq k \leq K$ are equal to one another, but unknown.  We show that it is also possible to estimate $\sigma$ at the fusion center with infrequent transmission of 1-bit messages from the sensors, which  allows one to obtain asymptotic confidence intervals for the estimators of $\mu$ that satisfy (\ref{CLT}). 

The proofs of these properties rely on classical, renewal-theoretic results, such as the asymptotic behavior of the overshoots, as well as on the
low-rate extensions of (\ref{rate1}), (\ref{rate2}) and (\ref{law}), which we described in Subsection \ref{sub11}.

\subsection{Summary}
The rest of the paper is organized as follows:  in Section \ref{sec2}, we obtain general, low-rate extensions of classical, renewal-theoretic results.
 More specifically, in Subsection \ref{sub21} we establish (\ref{rate2}) 
 for an arbitrary random walk whose drift $\delta$ goes to infinity so that  $\delta=o(t)$.  In Subsection \ref{sub22} we establish (\ref{rate1}) 
 for  an arbitrary renewal process whose rate $1/\delta$ goes to 0 so that $\delta=o(t)$.  Moreover, we illustrate these two theorems in the case that $\tau$ is stochastically less variable than an exponential random variable and in the case that $\tau$ is the first hitting time of a spectrally negative L\'evy process.  In Subsection \ref{sub23}, we establish a low-rate version of 
Anscombe's theorem.

In Section \ref{sec3}, we consider the problem of estimating the drift $\mu$ of a random walk, $(S_{t})_{t \in \mathbb{N}}$, where 
$\Exp[S_{t}]=\mu t$ and $\Var[S_{t}]=\sigma^{2} t$. In Subsection \ref{sub31} we assume that $(S_{t})_{t \in \mathbb{N}}$  is observed at an arbitrary renewal process, $(\tau_{n})_{n \in \mathbb{N}}$. We show that the asymptotic distribution of the sample mean, $S_{t}/t$, is preserved when we replace the current time $t$ with the last sampling time, $\tau(t)$, \textit{even if $\delta \rightarrow \infty$ so that $\delta=o(t)$}, as long as
$\Var[\tau]=\calo(\delta^{2})$.  In Subsection \ref{sub32}, 
we focus on the case that $(S_{t})$ is being sampled whenever it changes  by $\D>0$ since the previous sampling instance,  but the size of the overshoot is not known. 
We show that, even in this setup, the asymptotic distribution of $S_{t}/t$ can be preserved, 
as long as the sampling rate is sufficiently low, in the sense that $\D \rightarrow \infty$ so that $\sqrt{t} <<\D<<t$. 
In Subsection \ref{sub33} we construct improved estimators using classical results from renewal theory that allow us to approximate the unobserved overshoots, as long as we know the distribution of the random walk and it satisfies certain assumptions. In Subsection \ref{sub34}, we compare numerically the above
estimators and illustrate the advantages of low-rate sampling.  In Subsection \ref{sub35} we consider the estimation of $\sigma$.

Section \ref{sec3} focuses on the one-dimensional $(K=1)$ version of the estimation problem that we described in Subsection \ref{sub13}. We do so mainly in order to lighten the notation. In Section \ref{sec4}, we formulate these results in a multi-dimensional setup. We conclude  in Section \ref{sec5}. 

\section{Low-rate renewal theory} \label{sec2}
\subsection{The case of random walks} \label{sub21}
Let $(\tau_{n})_{n \in \mathbb{N}}$ be a random walk with positive drift $\delta$, i.e., $(\tau_{n}-\tau_{n-1})_{n \in \mathbb{N}}$ is  a sequence of i.i.d. random variables, not necessarily positive, with mean $\delta:=\Exp[\tau]\in (0,\infty)$ and finite variance $\Var[\tau]<\infty$, 
where $\tau:=\tau_{1}$, $\tau_{0}:=0$. For any $t>0$, the first-passage time $\nu(t) :=\inf\{n \in \mathbb{N}: \tau_{n} > t\}$  is a stopping time with respect to the filtration generated by $(\tau_{n})_{n \in \mathbb{N}}$, therefore, 
$$ 
\tau_{\nu(t) } -\delta \,  \nu(t)  = \sum_{n=1}^{\nu(t)} \Bigl(\tau_{n}-\tau_{n-1}-\delta \Bigr)
$$ 
is a \textit{stopped} random walk with drift zero and from Wald's identities we  have 
\begin{align} \label{wald} 
&\Exp[\tau_{\nu(t) } -\delta \,  \nu(t)] =0 , \quad \Var[\tau_{\nu(t) } -\delta \, \nu(t) ] = \Exp[\nu(t)] \;  \Var[\tau].
\end{align}
Moreover, from Lorden's \cite{lorden}  bound on the excess over a boundary we have
\begin{equation} \label{newlor1}
\sup_{t \geq 0}  \Exp[\tau_{\nu(t) }-t]  \leq  \frac{\Exp[(\tau^{+})^{2}]}{\delta}.
\end{equation}
Using (\ref{wald})-(\ref{newlor1}), we can establish a low-rate extension of (\ref{rate2}) when $r=1$. 

\begin{theorem} \label{theo:0}
If  $\Var[\tau]=\calo(\delta^{2})$ as $\delta \rightarrow \infty$,  
then
\begin{equation*}
\frac{\delta \nu(t) }{t} \overset{\cL^{1}}  \longrightarrow 1  \quad \text{as} \quad \delta,t \rightarrow \infty \quad \text{so that} \quad \delta=o(t).
\end{equation*}
When, in particular,  $\Var[\tau]=\calo(\delta^{q})$ as $\delta \rightarrow \infty$ for some $q \in [0,2]$,
\begin{equation} \label{details0}
\Exp\Bigl[\Big|\frac{\delta \nu(t) }{t}-1\Big|\Bigr] = \calo \Bigl( \frac{\delta}{t} \Bigr)
+ \calo\Bigl( \sqrt{\frac{\delta^{q-1}}{t}}\Bigr).
\end{equation}
\end{theorem}

\begin{proof} It clearly suffices to prove (\ref{details0}). Let us first observe that
\begin{align}\label{triarw}
\Big|\frac{\delta \nu(t) }{t}-1\Big| &\leq \frac{|\delta \nu(t) -\tau_{\nu(t) }|}{t} + \frac{\tau_{\nu(t) }-t}{t}.
\end{align}
From  (\ref{newlor1}) and the assumption  that $\Exp[(\tau^{+})^{2}]=\calo(\delta^{2})$ we have
\begin{align} \label{show10}
\begin{split}
\frac{\Exp[\tau_{\nu(t) }-t]}{t} \leq \frac{\Exp[(\tau^{+})^{2}]}{\delta \, t} = \frac{  \calo(\delta^{2})}{\delta \, t} =
\calo \Bigl( \frac{\delta}{t} \Bigr).
\end{split}
\end{align}
Moreover, starting with an application of H\"older's inequality, we have
\begin{align} \label{show20}
\frac{\Exp[|\delta \nu(t) -\tau_{\nu(t) }|]}{t} &\leq \sqrt{ \frac{\Exp[(\delta \nu(t) -\tau_{\nu(t) })^{2}]}{t^{2}}} 
= \sqrt{ \frac{\Exp [\nu(t) ] \Var[\tau]}{t^{2}}} = \sqrt{ \frac{\Exp [\tau_{\nu(t) }]}{t} \, \frac{\Var[\tau]}{\delta \, t}} \nonumber \\
&= \sqrt{ \Bigl(1+ \calo \Bigl( \frac{\delta}{t} \Bigr) \Bigr) \, \frac{\calo(\delta^{q})}{\delta t}} 
=\sqrt{\calo \Bigl( \frac{\delta^{q-1}}{t} \Bigr)}.
\end{align} 
The first two equalities follow from Wald's identities (\ref{wald}), whereas the third equality follows from (\ref{show10}) 
and the assumption that $\Var[\tau]=\calo(\delta^{q})$. Now, taking expectations in (\ref{triarw}) and applying (\ref{show10})-(\ref{show20})  
completes the proof. 
\end{proof}

\noindent \underline{\textbf{Remarks:}} (1) The proof of Theorem \ref{theo:0} provides an alternative way to prove (\ref{rate2}) (for $r=1$) when $\delta$ is fixed, under the additional assumption that $\Var[\tau]<\infty$.  

(2) The speed of convergence in  $\delta \nu(t) /t \overset{\cL^{1}} \rightarrow 1$ is determined by $q$, which describes the growth of $\Var[\tau]$ 
as $\delta \rightarrow \infty$. More specifically, the right-hand side in (\ref{details0}) is  $\calo(\sqrt{\delta/t})$ when $q=2$ 
 and $\calo(\delta/t)$ when $q=1$. 

In order to extend (\ref{rate2}) for arbitrary $r >1$, in addition to (\ref{wald})-(\ref{newlor1}) we need the
stopped versions of the Marcinkiewicz-Zygmund inequalities (see \cite{gut},  p.  22), according to which 
\begin{equation} \label{waldp}
\Exp[(\tau_{\nu(t) } - \delta \nu(t) )^{r}] \leq c_{r} \,  \Exp[\nu(t) ^{r/2}] \,  \Exp[|\tau-\delta|^{r}],
\end{equation}
for any $r \geq 2$, as well as the following generalization of (\ref{newlor1}), 
\begin{equation} \label{newlor}
\sup_{t \geq 0} \Exp[(\tau_{\nu(t) }-t)^{r}]  \leq  c_{r}  \, \frac{\Exp[(\tau^{+})^{r+1}]}{\delta},
\end{equation}
which is also due to Lorden \cite{lorden} and is valid for any $r\geq 1$. Here, and in what follows, $c_{r}$ represents a generic, positive constant that depends on $r$. Finally, we will need the following algebraic inequality:
\begin{equation} \label{alge}
|x+y|^{r} \leq 2^{r-1} \, (x^{r}+y^{r}), \quad x, y>0.
\end{equation}

\begin{theorem} \label{theo:1}
Let $r \geq 1$ and suppose that as $\delta \rightarrow \infty$
\begin{enumerate}
\item[(*)] $\quad \Exp[|\tau-\delta|^{r \vee 2}] =\calo(\delta^{r \vee 2})   \quad \text{and} \quad  \Exp[(\tau^{+})^{r+1}]=\calo(\delta^{r+1})$.
\end{enumerate}
Then $$
\frac{\delta \nu(t) }{t} \overset{\cL^{r}}  \longrightarrow 1 \quad \text{as} \quad \delta,t \rightarrow \infty \quad \text{so that} \quad \delta=o(t).
$$
If, in particular, $\Exp[|\tau-\delta|^{r \vee 2}] =\Theta(\delta^{q_{r \vee 2}})$, where  $q_{r \vee 2} \in [0, r \vee 2]$, then
\begin{equation} \label{details}
\Exp\Bigl[\Big|\frac{\delta \nu(t) }{t}-1\Big|^{r} \Bigr] = \calo \Bigl( \frac{\delta}{t} \Bigr)^{r}
+ \calo\Bigl(\frac{\delta^{\alpha_{r}}}{t}\Bigr)^{\frac{r}{2}} ,
\end{equation}
where $\alpha_{r}:= 2 (q_{r \vee 2} / r \vee 2)-1$.
\end{theorem}

It is clear that Theorem \ref{theo:1} reduces to Theorem \ref{theo:0} when $r=1$. Before we prove it, we need to state the following lemma, which 
implies that  if condition $(*)$ holds for some $r > 1$, then it also holds for any $s \in [1,r]$.

\begin{lemma} \label{lemC}
\begin{enumerate}
\item[(i)] If  $r \geq 1$, then
$$\Exp[(\tau^{+})^{r+1}]= \calo(\delta^{r+1}) \Rightarrow \Exp[(\tau^{+})^{s}] =\calo(\delta^{s}) , \quad \forall s \in [1, r+1]. $$
\item[(ii)] If $r \geq 2$, then 
$$\Exp[|\tau-\delta|^{r}] =\calo(\delta^{r}) \Rightarrow \Exp[|\tau-\delta|^{s}] =\calo(\delta^{s}),  \quad \forall s \in [2, r]. $$
\end{enumerate}
\end{lemma}

\begin{proof}
If  $\Exp[(\tau^{+})^{r+1}]= \calo(\delta^{r+1})$ for some $r \geq 1$, then from  H\"older's inequality it follows that   for any $s \in [1,r+1]$
$$\Exp[(\tau^{+})^{s}]  \leq  (\Exp[(\tau^{+})^{r+1}])^{s/(r+1)}= (\calo(\delta^{r+1}))^{s/(r+1)} = \calo(\delta^{s}).$$
Similarly, if $\Exp[|\tau-\delta|^{r}] =\calo(\delta^{r})$ for some $r \geq 2$, then 
$$\Exp[|\tau-\delta|^{s}] \leq  (\Exp[|\tau-\delta|^{r}])^{s/r}= (\calo(\delta^{r}))^{s/r}= \calo(\delta^{s}).$$
\end{proof}

\begin{proof}[Proof of Theorem \ref{theo:1}] 
Let $r \ge 1$. By definition, $\alpha_{r} \leq 1$, thus, it suffices to prove (\ref{details}). Applying the algebraic inequality (\ref{alge})
to (\ref{triarw}) and  taking expectations we have
\begin{align} \label{show0}
\Exp \Bigl[\Big|\frac{\delta \nu(t) }{t}-1\Big|^{r} \Bigr] &\leq c_{r} \Bigl( \frac{\Exp[|\delta \nu(t) -\tau_{\nu(t) }|^{r}]}{t^{r}} + \frac{\Exp[(\tau_{\nu(t) }-t)^{r}]}{t^{r}} \Bigr).
\end{align}
From (\ref{newlor}) and the assumption  $\Exp[(\tau^{+})^{r+1}]=\calo(\delta^{r+1})$ it follows that
\begin{align} \label{show1}
\begin{split}
\frac{\Exp[(\tau_{\nu(t) }-t)^{r}]}{t^{r}} \leq c_{r} \, \frac{\Exp[(\tau^{+})^{r+1}]}{\delta \, t^{r}} = \frac{  \calo(\delta^{r+1})}{\delta \, t^{r}} =
\calo \Bigl( \frac{\delta^{r}}{t^{r}} \Bigr).
\end{split}
\end{align}
From  H\"older's inequality and (\ref{waldp}) we have
\begin{align*} 
\Exp[|\delta \nu(t) -\tau_{\nu(t) }|^{r}]&\leq  \Exp[|\delta \nu(t) -\tau_{\nu(t) }|^{r \vee 2}]^{\frac{r}{r \vee 2}}\\
& \leq c_{r} \, \Bigl(\Exp [\nu(t) ^{\frac{r \vee 2}{2}} ]  \, \Exp[|\tau-\delta|^{r \vee 2}] \Bigr)^{\frac{r}{r \vee 2}}\\
&=  c_{r} \, \Exp\Bigl[\nu(t) ^{\frac{r \vee 2}{2}} \Bigr]^{\frac{r}{r \vee 2}}  \, \Exp\Bigl[|\tau-\delta|^{r \vee 2} \Bigr]^{\frac{r}{r \vee 2}},
\end{align*} 
and, consequently, 
\begin{align*}
\frac{\Exp[|\delta \nu(t) -\tau_{\nu(t) }|^{r}]}{t^{r}} & \leq   c_{r} \, \Exp \Bigl[ \Bigl(\frac{\delta \nu(t) }{t}\Bigr)^{\frac{r \vee 2}{2}} \Bigr] ^{\frac{r}{r \vee 2}}  \;  \frac{\Exp[|\tau-\delta|^{r \vee 2}]^{\frac{r}{r \vee 2}}}{(\delta t)^{\frac{r}{2}}}
\end{align*}
Therefore, from the assumption  $\Exp[|\tau-\delta|^{r \vee 2}] =\calo(\delta^{q_{r \vee 2}})$ and the definition of $\alpha_{r}$ it follows that
\begin{align}  \label{show2}
\frac{\Exp[|\delta \nu(t) -\tau_{\nu(t) }|^{r}]}{t^{r}} &\leq \Exp \Bigl[ \Bigl(\frac{\delta \nu(t) }{t}\Bigr)^{\frac{r \vee 2}{2}} \Bigr] ^{\frac{r}{r \vee 2}} 
  \calo\Bigl(\frac{\delta^{a_{r}}}{t}\Bigr)^{\frac{r}{2}}.
\end{align}
From (\ref{show0})-(\ref{show2}) we obtain
$$\Exp \Bigl[\Big|\frac{\delta \nu(t) }{t}-1\Big|^{r} \Bigr] \leq \calo \Bigl( \frac{\delta}{t} \Bigr)^{r}+ 
 \Exp \Bigl[ \Bigl(\frac{\delta \nu(t) }{t}\Bigr)^{\frac{r \vee 2}{2}} \Bigr] ^{\frac{r}{r \vee 2}}  \; \calo\Bigl(\frac{\delta^{a_{r}}}{t}\Bigr)^{\frac{r}{2}}
$$
and in order to complete the proof it suffices to show that for any $r \geq 1$,
$$\Exp \Bigl[ \Bigl(\frac{\delta \nu(t) }{t}\Bigr)^{\frac{r \vee 2}{2}} \Bigr] =1+o(1) \quad \text{as} \quad \delta,t  \rightarrow \infty \quad \text{so that} \quad \delta=o(t).$$

More specifically, when $r \in [1,2]$, it suffices to show that  if $\Var[\tau]=\Theta(\delta^{q_{2}})$ for some $q_{2} \in [0,2]$ and $\Exp[(\tau^{+})^{r+1}]=\calo(\delta^{r+1})$, then $$ \Exp \Bigl[\frac{\delta \nu(t) }{t} \Bigr]\rightarrow 1 
\quad \text{as} \quad \delta,t  \rightarrow \infty \quad \text{so that} \quad \delta=o(t).$$
This follows directly from  Theorem \ref{theo:0} and  Lemma \ref{lemC}(i). Therefore, (\ref{details}) holds for any $1 \leq r \leq 2$.  When  $r \in [2,4]$, it suffices to show that if $\Exp[|\tau-\delta|^{r}] =\calo(\delta^{q_{r}})$ for some $q_{r} \in [0, r]$ and $\Exp[(\tau^{+})^{r+1}]=\calo(\delta^{r+1})$, then 
$$\Exp \Bigl[ \Bigl(\frac{\delta \nu(t) }{t} \Bigr)^{\frac{r}{2}}\Bigr]\rightarrow 1 \quad \text{as} \quad \delta,t  \rightarrow \infty \quad \text{so that} \quad \delta=o(t).$$ 
This follows from  Lemma \ref{lemC} and the fact that  (\ref{details}) holds for $r/2 \in (1,2)$. This proves (\ref{details}) for $r \in [2,4]$ and the proof is similar when  $r \in [4,8]$,  etc.
\end{proof}

\noindent \underline{Remarks:}

(1) The proof of Theorem \ref{theo:1} provides an alternative way to prove (\ref{rate2}) in the case that $\delta$ is fixed, under the 
additional condition that  $\Exp[(\tau^{+})^{r+1}]< \infty$.  \\

(2) The rate of convergence in $\delta \nu(t) /t \overset{\cL^{r}} \rightarrow 1$ is determined by $q_{r \vee 2}$, which describes the growth of the $r \vee2$-
central moment  of $\tau$ as $\delta \rightarrow \infty$. Specifically, the right-hand side in (\ref{details}) is equal to $\calo(\sqrt{\delta^{r}/t^{r}})$ when $q_{r \vee 2}=r \vee 2$ and   $\calo(\delta^{r}/t^{r})$ when  $q_{r \vee 2}=(r \vee 2)/2$. \\ 

(3) There is a clear dichotomy in the conditions required for $\delta \nu(t)/t \overset{\cL^{r}}  \longrightarrow 1$ to hold as $\delta, t \rightarrow \infty$ so that $\delta=o(t)$. Indeed, when $1 \leq r \leq 2$, 
the \textit{second} central moment of $\tau$  must be finite and of order $\calo(\delta^{2})$; when $r \geq 2$, 
the $r^{th}$ central moment of $\tau$ must be finite and of order $\calo(\delta^{r})$. Compare, in particular, Theorem \ref{theo:0} with the 
following  corollary of Theorem \ref{theo:1}.

\begin{corollary} \label{newcoro}
If   $\Var[\tau]=\calo(\delta^{2})$ and $\Exp[(\tau^{+})^{3}]=\calo(\delta^{3})$  as  $\delta \rightarrow \infty$, then  
$$
\frac{\delta \, \nu(t) }{t} \overset{\cL^{2}}  \longrightarrow 1 \quad \text{as} \quad \delta,t \rightarrow \infty \quad \text{so that} \quad \delta=o(t).
$$
When, in particular,  $\Var[\tau]=\calo(\delta^{q})$ for some $q \in [0,2]$, then as $\delta,t \rightarrow \infty$ so that $\delta=o(t)$:
$$
\Exp\Bigl[\Big|\frac{\delta \nu(t) }{t}-1\Big|^{2}\Bigr] = \calo \Bigl( \frac{\delta}{t} \Bigr)^{2}
+ \calo\Bigl(\frac{\delta^{q-1}}{t}\Bigr).
$$
\end{corollary}

\begin{proof}
If follows by setting $r=2$ in Theorem \ref{theo:1}.
\end{proof}

\subsection{The case of renewal processes} \label{sub22}
We now extend  (\ref{rate1}) 
to the case that $\delta,t \rightarrow \infty$ so that $\delta=o(t)$.

\begin{theorem} \label{coro:1}
Let $r \geq 1$ and suppose that $\Pro(\tau>0)=1$. If  $\Exp[\tau^{r+1}]=\calo(\delta^{r+1})$ as $\delta \rightarrow \infty$, then 
$$
\frac{\delta N(t) }{t} \overset{\cL^{r}}  \longrightarrow 1  \quad \text{as} \quad \delta,t \rightarrow \infty \quad \text{such that} \quad \delta=o(t).
$$
When, in particular, $\Exp[|\tau-\delta|^{r \vee 2}] =\calo(\delta^{q_{r \vee 2}})$, where  $q_{r \vee 2} \in [0, r \vee 2]$, then
\begin{equation} \label{details2}
\Exp\Bigl[\Big|\frac{\delta N(t) }{t}-1\Big|^{r} \Bigr] = \calo \Bigl( \frac{\delta}{t} \Bigr)^{r}
+ \calo\Bigl(\frac{\delta^{\alpha_{r}}}{t}\Bigr)^{\frac{r}{2}} ,
\end{equation}
where $\alpha_{r}:= 2 (q_{r \vee 2} / r \vee 2)-1$.
\end{theorem}

\begin{proof}
It clearly suffices to show (\ref{details2}). Since $\nu(t) =N(t) +1$ when $\Pro(\tau>0)=1$,
we have  
\begin{align*}
\Big|\frac{\delta N(t) }{t}-1\Big| &\leq  \Big|\frac{\delta \nu(t) }{t}-1\Big|  + \frac{\delta}{t},
\end{align*}
thus, applying the algebraic inequality (\ref{alge}) and taking expectations we obtain
$$
\Exp\Bigl[ \Big|\frac{\delta N(t) }{t}-1\Big|^{r} \Bigr] \leq c_{r} \;  \Bigl( \Exp \Bigl[ \Big|\frac{\delta \nu(t) }{t}-1\Big|^{r} \Bigr]  + \Bigl(\frac{\delta}{t}\Bigr)^{r} \Bigr).
$$
Therefore, it  suffices to show that condition (*) of Theorem \ref{theo:1} is implied by $\Exp[\tau^{r+1}]=\calo(\delta^{r+1})$ when $\Pro(\tau>0)=1$. 
Indeed, for any $r \geq 1$ we have
$$\Exp[|\tau-\delta|^{r \vee 2}] \leq \Exp[\tau^{r \vee 2}]  \leq \Exp[\tau^{r+1}]^{\frac{r\vee 2}{r+1}}=\calo(\delta^{r+1})^{\frac{r\vee 2}{r+1}}= \calo(\delta^{r \vee 2}),$$
where the first inequality is due to $\Pro(\tau>0)=1$ and the  second one follows from H\"older's inequality and the fact that $r \vee 2 \leq r+1$.    
\end{proof}

We now illustrate the previous theorems for two classes of renewal processes.

\subsubsection{Renewals stochastically less variable than the Poisson process}  

Let us start by showing that the  condition of Theorem \ref{coro:1}  is satisfied when $\tau$ is exponentially distributed. Indeed, for any $r \geq 1$, 
if we denote by  $\lceil r \rceil :=\min\{n \in \mathbb{N}: n \geq r\}$ the ceiling of $r$, then
\begin{align*}
\Exp[\tau^{r+1}] & \leq \Exp[\tau^{\lceil r+1 \rceil}]^{\frac{r+1}{ \lceil r+1 \rceil}} \leq \delta^{r+1},
\end{align*} 
where the  first  inequality follows from H\"older's inequality and the second one holds  because $\Exp[\tau^{n}]= \delta^{n} / n! \leq \delta^{n}$ for any $n \in \mathbb{N}$. 

More generally, suppose that  $\tau$ is not itself exponentially distributed, but  is \textit{stochastically less variable} than an exponential random variable $\tau'$  that has mean $\delta$, i.e., 
$$\Exp[g(\tau)] \leq \Exp[g(\tau')] \; \forall \; \text{convex and increasing function} \; g: (0, \infty) \rightarrow (0, \infty).$$ 
Then, since $x \rightarrow x^{r+1}$ is convex and increasing  on $(0, \infty)$, it follows that  $\Exp[\tau^{r+1}] \leq \Exp[(\tau')^{r+1}] \leq  \delta^{r+1}$, which proves that the  condition of Theorem \ref{coro:1} is still satisfied when $\tau$ is stochastically less variable than an exponential random variable.  This is, for example,  the case when $\tau$ is \textit{new better than used in expectation}, i.e., 
$\Exp[\tau-t|\tau>t] \leq \Exp[\tau]$ for every $t>0$   (see, e.g., \cite{ross}, pp. 435-437).  

\subsubsection{First-hitting times of spectrally negative L\'evy process}

Suppose that $\tau= \inf\{t: Y_{t} \geq \D \}$, where $\D>0$  and $(Y_{t})_{t \geq 0}$ is a spectrally negative L\'evy process with positive drift. That is, 
$(Y_{t})_{t \geq 0}$ is a stochastic process that is continuous in probability, has stationary and independent increments, does not have positive jumps and 
$\Exp[Y_{t}]=\mu \, t$  for some $\mu>0$. 
Then,  from Wald's identity it follows that $\delta=\Exp[\tau]=\D/\mu$ and if, additionally, $\Exp[(Y_{1}^{-})^{r+1}]<\infty$, then from Theorem 4.2 in \cite{gut2} we have  
$$
\Exp[\tau^{r+1}]=(\D /\mu)^{r+1} (1+o(1)) \quad \text{as}  \quad \D \rightarrow \infty.
$$
Note that $\delta \rightarrow \infty$ if either $\D \rightarrow \infty$ or $\mu \rightarrow 0$. Thinking of $(\tau_{n})$ as  sampling times of $(Y_{t})$, it is natural to consider $\mu$ as fixed and to assume that the sampling period $\delta$ is controlled by threshold $\D$. This implies that $\Exp[\tau^{r+1}]=\Theta(\delta^{r+1})$ as $\delta \rightarrow \infty$ and proves that the condition of Theorem \ref{coro:1} is satisfied.

Furthermore, if $\Var[Y_{t}]=\sigma^{2} \, t$ for some $\sigma>0$, then $\Var[\tau]= (\sigma^{2}/\mu^{3}) \, \D$, which implies that $\Var[\tau]=\Theta(\delta)$ as $\delta \rightarrow \infty$. Therefore, 
from Theorem \ref{coro:1} we obtain
$$
\Exp\Bigl[\Big|\frac{\delta N(t) }{t}-1\Big| \Bigr] = \calo \Bigl( \frac{\delta}{t} \Bigr).
$$

\noindent \underline{\textbf{Remark:}} Analogous results can be obtained when $(Y_{t})_{t \in \mathbb{N}}$ is a random walk and/or $\tau= \inf\{t: |Y_{t}|
\geq \D \}$, a setup that we consider in detail in Subsection \ref{sub32}.

\subsection{A low-rate version of Anscombe's theorem} \label{sub23}
Let $(Y_{t})_{t \in \mathbb{N}}$ be a sequence of random variables that is uniformly continuous in probability  and converges in distribution as $t  \rightarrow \infty$ to a random variable, $Y_{\infty}$, i.e.,  $Y_{t}  \overset{\cD} \rightarrow Y_{\infty}$ as $t \rightarrow \infty$.
Suppose that $(Y_{t})_{t \in \mathbb{N}}$ can only be observed at a sequence of random times $(\tau_{n})_{n \in \mathbb{N}}$ that form a renewal process.  As before, we set $\tau:=\tau_{1}$,  $N(t):=\max\{n: \tau_{n} < t\}$ and $\tau(t):=\tau_{N(t)}$ is the most recent sampling time at time $t$.
In the following theorem we show that this convergence remains valid when we replace $t$  with the most recent sampling time, $\tau(t)$, 
even if $\delta  \rightarrow \infty$ so that $\delta=o(t)$, as long as the second moment of $\tau$ grows at most as the square of its mean,  $\delta^{2}$.

\begin{theorem} \label{lem:31}
If $\Exp[\tau^{2}]=\calo(\delta^{2})$, or equivalently $\Var[\tau]=\calo(\delta^{2})$,  then $\tau(t)/t \overset{\cL^{1}}  \rightarrow 1$ 
and, consequently, $Y_{\tau(t)} \overset{\cD}  \rightarrow Y_{\infty}$ as $\delta, t \rightarrow \infty$ so  that $\delta=o(t)$.
\end{theorem}

\begin{proof} 
From Lorden's  bound (see \cite{lorden}, p.526) on the expected value of the age of a renewal process and the assumption of the theorem we obtain
\begin{equation} \label{age}
\Exp[t-\tau(t)]  \leq \Exp[\tau^{2}] / \delta= \calo(\delta).
\end{equation} 
Therefore, 
$$\Exp\Bigl[\Big|\frac{\tau(t)}{t}-1\Big|\Bigr] =  \frac{\Exp[t-\tau(t)]}{t} =\calo\Bigl(\frac{\delta}{t}\Bigr),$$
which proves that $\tau(t)/t \overset{\cL^{1}}  \rightarrow 1$ (and, consequently,  $Y_{\tau(t)} \overset{\cD}  \rightarrow Y_{\infty}$, due to Anscombe's theorem) as  $\delta, t \rightarrow \infty$ so  that $\delta=o(t)$. Finally, since $\Exp[\tau^{2}]= \delta^{2}+ \Var[\tau]$, it is clear that $\Exp[\tau^{2}]=\calo(\delta^{2})$ is equivalent to $\Var[\tau]=\calo(\delta^{2})$. 
\end{proof}

\noindent \underline{\textbf{Remarks:}} 
(1) The condition of Theorem  \ref{lem:31} is satisfied when $(\tau_{n})$ is the Binomial process, i.e., sampling at any time $t \in \mathbb{N}$ has probability $1/\delta$ and is independent of the past. Indeed, in this case $\tau$ is geometrically distributed with mean $\delta$ and its variance is equal to  $\Var[\tau]= \delta^{2} (1 - \delta^{-1})$. \\ 

(2) Theorem \ref{lem:31} also applies when $(\tau_{n})$ corresponds to the hitting times of random walk, a setup that we consider in detail in Subsection \ref{sub32}.

\section{Low-rate estimation} \label{sec3}
In this section,  $(X_{t})_{t \in \mathbb{N}}$ is a sequence  of i.i.d. random variables with unknown mean $\mu:=\Exp[X]$ and finite standard deviation 
$\sigma:=\sd[X]$, where $X:=X_{1}$. 
From the Central Limit Theorem we know that the sample mean  is an asymptotically normal estimator of $\mu$, i.e., 
\begin{equation} \label{ans2}
\frac{\mu_{t}-\mu}{\sigma/\sqrt{t}} \longrightarrow \cN(0,1)  \quad \text{as} \quad t \rightarrow \infty,
\end{equation}
where $\mu_{t}:=S_{t}/t$ and $S_{t}:=X_{1}+ \ldots+X_{t}$.

\subsection{Sampling at an arbitrary renewal process} \label{sub31}
When the random walk $(S_{t})_{t \in \mathbb{N}}$ is observed only at a sequence of random times $(\tau_{n})_{n \in \mathbb{N}}$ that form a renewal process,
two  natural modifications of $\mu_{t}$ are 
\begin{equation} \label{bar}
\bar{\mu}_{t}= \frac{S_{\tau(t)}}{\tau(t)} \quad \text{and} \quad \tilde{\mu}_{t}= \frac{S_{\tau(t)}}{t}.
\end{equation}
That is, $\bar{\mu}_{t}$  is the sample mean evaluated at the most recent sampling time $\tau(t)=\tau_{N(t)}$ and is well defined only when $t \geq \tau_{1}$, in contrast to 
$\tilde{\mu}_{t}$ that is well defined for any $t>0$. The following theorem describes the  asymptotic behavior of these estimators when $\delta, t \rightarrow \infty$ so that $\delta=o(t)$. 

\begin{theorem}  \label{ans3}
Suppose that $\Var[\tau]=\calo(\delta^{2})$ as $\delta \rightarrow \infty$. 
\begin{enumerate}
\item[(i)] If  $\delta, t \rightarrow \infty$ so that  $\delta=o(t)$, then $\bar{\mu}_{t} \overset{p} \rightarrow \mu$ and 
\begin{equation} \label{clt1}
\frac{\bar{\mu}_{t}-\mu}{\sigma/\sqrt{t}} \rightarrow \cN(0,1).
\end{equation}
\item[(ii)] If  $\delta, t \rightarrow \infty$ so that  $\delta=o(t)$, then $\tilde{\mu}_{t}\overset{p} \rightarrow \mu$.
If also $\delta=o(\sqrt{t})$, then 
\begin{equation} \label{clt2}
\frac{\tilde{\mu}_{t}-\mu}{{\sigma/\sqrt{t}}} \rightarrow \cN(0,1).
\end{equation}
\end{enumerate}
\end{theorem}
\begin{proof}

(i) From the Strong Law of Large Numbers we know that $\mu_{t} \overset{a.s.} \rightarrow \mu$ as $t \rightarrow \infty$. Theorem \ref{lem:31} implies that $\tau(t)/t  \overset{p} \rightarrow 1$ and, consequently, $\tau(t) \overset{p} \rightarrow \infty$ and $\bar{\mu}_{t}  \overset{p} \rightarrow \mu$ as $\delta, t\rightarrow \infty$ so that $\delta=o(t)$ (see, e.g., \cite{gut}, page 12). From Theorem \ref{lem:31} and (\ref{ans2}) we obtain (\ref{clt1}).


(ii) We observe  that 
\begin{equation} \label{qaz}
|\bar{\mu}_{t} - \tilde{\mu}_{t}| =  |S_{\tau(t)}| \, \frac{t -\tau(t)}{t \,  \tau(t)} = |\bar{\mu}_{t}| \, \Bigl(1- \frac{\tau(t)}{t}\Bigr).
\end{equation}
Since  $\tau(t)/t  \overset{p} \rightarrow 1$ and $\bar{\mu}_{t} \overset{p} \rightarrow \mu$  as $\delta, t\rightarrow \infty$ so that $\delta=o(t)$,
(\ref{qaz}) clearly implies that  $\tilde{\mu}_{t} \overset{p} \rightarrow \mu$ as $\delta, t\rightarrow \infty$ so that $\delta=o(t)$.  Finally, in order to prove (\ref{clt2}),  it suffices to show that
$$
\sqrt{t} |\bar{\mu}_{t} - \tilde{\mu}_{t}| = |\bar{\mu}_{t}| \, \frac{t -\tau(t)}{\sqrt{t}}  \overset{p} \longrightarrow 0 \quad \text{as} \quad \delta,t \rightarrow \infty \; \text{so that} \; \delta=o(\sqrt{t}).
$$ This is indeed the case since  from (i) we have that $\bar{\mu}_{t} \overset{p} \rightarrow \mu$ when $\delta=o(t)$, whereas from (\ref{age}) it follows that $(t -\tau(t))/ \sqrt{t} \overset{\cL^{1}} \longrightarrow  0$ when $\delta=o(\sqrt{t})$.
\end{proof}

\noindent \underline{\textbf{Remark:}} When the sampling period $\delta$ is fixed as $t \rightarrow  \infty$,  it is clear that both  $\bar{\mu}_{t}$  and $\tilde{\mu}_{t}$  preserve the asymptotic distribution (\ref{ans2}). On the other hand, when $\delta \rightarrow \infty$ the asymptotic behavior of the two estimators differs, since $\tilde{\mu}_{t}$ requires that the sampling rate should not be too low  ($\delta << \sqrt{t}$), a  condition that is not necessary for 
$\bar{\mu}_{t}$. Therefore,  $\bar{\mu}_{t}$ and $\tilde{\mu}_{t}$ have similar behavior for small $\delta$, but  $\bar{\mu}_{t}$ is  more robust than $\tilde{\mu}_{t}$ under low-rate sampling, a conclusion that is also verified empirically in Figure \ref{fig:1}.

\subsection{Sampling at first hitting times} \label{sub32}
We now focus on the case that the random walk $(S_{t})_{t \in \mathbb{N}}$ is being sampled whenever it changes by a fixed amount $\D$ since the previous sampling instance, i.e., 
\begin{equation} \label{modelo}
\tau_{n}:= \inf\{t \geq \tau_{n-1}: S_{t}- S_{\tau_{n-1}} \geq \D \}, \quad n \in \mathbb{N},
\end{equation} 
where, for simplicity, we have assumed that $\mu>0$. 
Then, the estimators $\bar{\mu}_{t}$, $\tilde{\mu}_{t}$, defined in  (\ref{bar}),  take the form
\begin{align} 
\bar{\mu}_{t}&=\frac{1}{\tau(t)} \,  \sum_{n=1}^{N(t) } \Bigl[\D+\eta_{n}\Bigr] = \frac{N(t)  \, \D}{\tau(t)}+ \frac{1}{\tau(t)} \, \sum_{n=1}^{N(t)} \eta_{n} , \label{bar2} \\
 \tilde{\mu}_{t}&=\frac{1}{t} \,  \sum_{n=1}^{N(t) } \Bigl[\D+\eta_{n} \Bigr]= \frac{N(t)  \, \D}{t} + \frac{1}{t} \, \sum_{n=1}^{N(t)} \eta_{n}, \label{tilde2}
\end{align}
where $\eta_{n}:=S_{\tau_{n}}-S_{\tau_{n-1}}-\D$. Since $X$ has a finite second moment, the overshoots $(\eta_{n})$ are i.i.d. with $\Exp[\eta]=\calo(1)$ as $\D \rightarrow \infty$, 
where $\eta:=\eta_{1}$ (see, e.g., Theorem 10.5 in \cite{gut}). Therefore, from  Wald's identity it follows that 
\begin{equation} \label{walder}
\delta =\Exp[\tau]=\frac{\Exp[S_{\tau}]}{\mu}= \frac{\D +\Exp[\eta]}{\mu}
\end{equation} 
and, since we consider $\mu$ to be fixed, the sampling period $\delta$ is controlled by threshold $\D$, in the  sense that  
$\delta= \D/\mu +\calo(1)=\Theta(\D)$ as $\D \rightarrow \infty$.  Moreover, since
\begin{equation} \label{neon}
\Var[\tau]= (\sigma^{2}/\mu^{3}) \, \D \,  (1+o(1)) \quad  \text{as} \quad \D \rightarrow \infty,
\end{equation}
(see, e.g., Theorem 9.1 in \cite{gut}), it is clear that $\Var[\tau]=\Theta(\delta)$ as $\delta \rightarrow \infty$.
Therefore, as $\delta,t \rightarrow \infty$ so that $\delta=o(t)$,  from Theorems \ref{theo:1} and  \ref{coro:1} we have 
\begin{equation} \label{lr}
\Exp \Bigl[\Big|\frac{\delta \nu(t) }{t}-1\Big|\Bigr] =\calo(\delta/t) \quad \text{and} \quad
\Exp \Bigl[\Big|\frac{\delta N(t) }{t}-1\Big|\Bigr] =\calo(\delta/t),
\end{equation}
from Theorem \ref{lem:31} we obtain
\begin{equation} \label{taut}
\tau(t)/t \overset{p} \rightarrow 1,
\end{equation}
and we also observe that the  condition of Theorem \ref{ans3} is satisfied. However,  $\bar{\mu}_{t}$ and  $\tilde{\mu}_{t}$ are not applicable when the overshoots $(\eta_{n})_{n \in \mathbb{N}}$ are unobserved,  i.e., when at each time $\tau_{n}$ we do not learn the excess of $S_{\tau_{n}}- S_{\tau_{n-1}}$ over $\D$. In this case, $\bar{\mu}_{t}$ and  $\tilde{\mu}_{t}$ reduce to 
\begin{equation} \label{hatcjeck}
\hat{\mu}_{t}:= \frac{N(t)  \, \D}{\tau(t)} \qquad \text{and} \qquad  \check{\mu}_{t}:= \frac{N(t)  \, \D}{t},
\end{equation}
respectively. In the following theorem
we show that both $\hat{\mu}_{t}$ and $\check{\mu}_{t}$ are consistent when $t$ \textit{and} $\delta \rightarrow \infty$, but not for fixed $\delta$.
More surprisingly, we also show that if the sampling rate is sufficiently \textit{low} ($\sqrt{t} << \delta << t$), then $\hat{\mu}_{t}$ preserves the asymptotic distribution (\ref{ans2}). On the other hand, $\check{\mu}_{t}$ fails to do so for any sampling rate.

\begin{theorem} \label{theo:32}
\begin{enumerate}
\item[(i)] If $\delta,t \rightarrow \infty$ so that $\delta=o(t)$, then
$$
\Exp[|\check{\mu}_{t}-\mu|] = \calo \Bigl(\frac{\delta}{t} \Bigr)+ \Exp[\eta] \, \calo \Bigl(\frac{1}{\delta} \Bigr),
$$ 
and, consequently, $\sqrt{t} \, \Exp[|\check{\mu}_{t}-\mu|] = \calo(1)$ when $\delta=\Theta(\sqrt{t})$.

\item[(ii)] If $\delta,t \rightarrow \infty$ so that $\delta=o(t)$, then $\hat{\mu}_{t} \overset{p} \rightarrow \mu$.

\item[(iii)] If $\delta,t \rightarrow \infty$ so that $\sqrt{t} << \delta <<t$, i.e., $\delta=o(t)$ and $\sqrt{t}=o(\delta)$, then 
$$ \frac{\hat{\mu}_{t}-\mu}{\sigma/\sqrt{t}} \rightarrow \cN(0,1).
$$
\end{enumerate}
\end{theorem}


\begin{proof}
(i) From the definition of $\check{\mu}_{t}$ and  (\ref{walder}) we have 
\begin{align} \label{o}
\begin{split}
|\check{\mu}_{t}-\mu|=  \Big|\D \frac{N(t) }{t}- \mu \Big| &=  \Big|\Bigl(\delta \mu -\Exp[\eta] \Bigr) \, \frac{N(t) }{t}- \mu \Big| \\
&\leq \mu \Big| \frac{ \delta N(t) }{t}-1\Big| + \Exp[\eta] \frac{N _{t}}{t}.
\end{split}
\end{align}
Taking expectations and applying (\ref{lr}) proves  (i).  

(ii) Due to Theorem  \ref{ans3}(i), it suffices to show that 
\begin{equation} \label{w}
|\bar{\mu}_{t}-\hat{\mu}_{t}| = \frac{t}{\tau(t)} \, \frac{1}{t} \sum_{n=1}^{N(t) } \eta_{n} \overset{p} \longrightarrow 0,
\end{equation}
or equivalently, due to (\ref{taut}), that $(\sum_{n=1}^{N(t) } \eta_{n})/t \overset{p} \rightarrow 0$  as $\delta, t \rightarrow \infty$  so that $\delta=o(t)$.
Indeed,  
\begin{align} \label{e}
\frac{1}{t} \Exp\Bigl[\sum_{n=1}^{N(t) } \eta_{n}\Bigr] &\leq \frac{1}{t} \Exp\Bigl[\sum_{n=1}^{\nu(t) } \eta_{n}\Bigr] 
= \frac{ \Exp[\nu(t) ] \Exp[\eta]  }{t} = \calo \Bigl(\frac{1}{\delta}\Bigr),
\end{align}
where the inequality is due to $\nu(t) =N(t) +1$, the first equality follows from Wald's first identity and the second one from 
(\ref{lr}) and the fact that $\Exp[\eta]=\calo(1)$ as $\D \rightarrow \infty$.

(iii) 
Due to Theorem \ref{ans3}(i),  it clearly suffices to show that $\sqrt{t} |\bar{\mu}_{t}-\hat{\mu}_{t}| \overset{p} \rightarrow 0$ or equivalently,
due to (\ref{taut}) and (\ref{w}), that $(\sum_{n=1}^{N(t) } \eta_{n})/\sqrt{t} \overset{p} \rightarrow 0$ as $\delta, t \rightarrow \infty$  so that $\sqrt{t}<< \delta  << t$,
which follows directly from (\ref{e}).
 \end{proof}

\noindent \underline{\textbf{Remarks:}} (1) Theorem \ref{theo:32} shows that a sufficiently \textit{low} sampling rate (or, equivalently, a sufficiently large threshold $\D$) is needed for $\hat{\mu}_{t}$ to  preserve the asymptotic distribution (\ref{ans2}) and for $\check{\mu}_{t}$ to be $\sqrt{t}$-consistent. This is quite intuitive, since small values of $\D$ may lead to frequent sampling, but they also lead to fast accumulation of large unobserved overshoots, thus they intensify the related performance loss. On the contrary, large thresholds guarantee that the relative size of the overshoots will be small, mitigating in this way the corresponding performance loss.     \\

\noindent (2) Theorem \ref{theo:32} and the discussion prior to it remain valid, with obvious modifications, when the sign of $\mu$ is unknown, as long as 
(\ref{modelo}) and (\ref{hatcjeck}) are replaced by
\begin{align*}
&\tau_{n}= \inf\{t \geq \tau_{n-1}: |S_{t}- S_{\tau_{n-1}}| \geq \D \}; \; \tau_{0}=0 \quad \text{and}  \\
&\bar{\mu}_{t}= \frac{1}{\tau(t)} \, \sum_{n=1}^{N(t)} (2z_{n}-1) \quad \text{and} \quad   \tilde{\mu}_{t}=\frac{1}{t} \,  \sum_{n=1}^{N(t)} (2 z_{n}-1),
\end{align*}
respectively, where $z_{n}:=1(S_{\tau_{n}}- S_{\tau_{n-1}} \geq \D)$.

\subsection{Efficient estimation via  overshoot correction} \label{sub33}
The estimators $\hat{\mu}_{t}$ and $\check{\mu}_{t}$ do not require knowledge of the distribution of $X$. However, when this distribution is known (up to the unknown parameter $\mu$), it should be possible to improve $\bar{\mu}_{t}$ and $\tilde{\mu}_{t}$ by approximating, instead of ignoring, the unobserved overshoots.
In order to achieve that, a first idea is to replace each $\eta_{n}$ in (\ref{bar2}) and (\ref{tilde2}) by its expectation, $\Exp[\eta]$. However, since the latter is typically an intractable quantity, we could use the \textit{limiting average overshoot} instead. Indeed, it is well known from classical renewal theory 
(see, e.g., \cite{gut}, p.  105) that if $X$ is non-lattice, then 
 \begin{equation} \label{asyover}
 \lim_{\D \rightarrow \infty} \Exp[\eta]= \rho(\mu):= \frac{\Exp[S_{\tau_{+}}^{2}]}{2 \, \Exp[S_{\tau_{+}}]},
 \end{equation}
 where $\tau_{+}:=\inf\{t \in \mathbb{N}: S_{t}>0\}$ is the first ascending ladder time (and, consequently, $S_{\tau_{+}}$ is the first ascending ladder height).
Note that we have expressed the limiting average overshoot as a function of $\mu$ in order to emphasize that it depends on the unknown parameter, $\mu$,
thus, it cannot be used directly to approximate the unobserved overshoots. However, we can obtain a working approximation of $\rho(\mu)$ if we replace 
$\mu$ with an estimator that does not require knowledge of the distribution of $X$, such as $\check{\mu}_{t}$ or $\hat{\mu}_{t}$. Doing so, we obtain 
\begin{align}
\begin{split}
\frac{1}{t} \, \sum_{n=1}^{N(t) }\Bigl[\D+ \rho(\check{\mu}_{t})\Bigr] & =\frac{N(t)  \, [\D + \rho(\check{\mu}_{t})]}{t} 
=:g(\check{\mu}_{t}) \\
\frac{1}{\tau(t)} \, \sum_{n=1}^{N(t) }\Bigl[\D+ \rho(\hat{\mu}_{t})\Bigr] &= \frac{N(t)  \, [\D + \rho(\hat{\mu}_{t})]}{\tau(t) } 
=:g(\hat{\mu}_{t}),
\end{split}
\end{align}
where $g(x):=x \,(1+ \rho(x)/\D)$, $x>0$. Note that the factor in the parenthesis reflects the overshoot correction that is achieved by the suggested approximation. In the next theorem we show that,  under certain conditions on the distribution of $X$, $g(\check{\mu}_{t})$, unlike $\check{\mu}_{t}$, can preserve the asymptotic distribution (\ref{ans2}), whereas $g(\hat{\mu}_{t})$ attains it for a wider range of sampling rates than $\hat{\mu}_{t}$.

\begin{theorem} \label{theo:33}
Suppose that $X$ is non-lattice and 
\begin{enumerate}
\item[(A1)] $\Exp[|X|^{r+1}]<\infty$ for some $r \geq  2$, 
\item[(A2)] $|\Exp[\eta]-\rho(\mu)| =\calo(\D^{-(r-1)})$ as $\D \rightarrow \infty$ for some $r \geq 2$,
\item[(A3)] $\mu \rightarrow \rho(\mu)$ is Lipschitz function.
\end{enumerate}
Then, (i) as  $\delta, t \rightarrow \infty$ so that $t^{1/4} << \delta <<\sqrt{t}$
$$ \frac{g(\check{\mu}_{t})-\mu}{\sigma/\sqrt{t}} \rightarrow \cN(0,1)$$
and (ii) as  $\delta, t \rightarrow \infty$ so that $t^{1/4} << \delta <<t$
$$\frac{g(\hat{\mu}_{t})-\mu}{\sigma/\sqrt{t}}  \rightarrow \cN(0,1).$$
\end{theorem}

Before we prove this theorem, let us first comment on its assumptions. (A2) describes how fast the  expected overshoot should converge 
to its limiting value as $\D  \rightarrow \infty$. We will show in Subsection \ref{subA2} that (A2) is implied by (A1)  when $X>0$ or by an exponentially decaying right tail of $X$.  (A3)  guarantees that if  $\check{\mu}_{t}$ (or $\hat{\mu}_{t}$) is a ``good'' estimator of $\mu$, then so will    $\rho(\check{\mu_{t}})$ (or $g(\hat{\mu}_{t})$) be for $\rho(\mu)$. Sufficient conditions for (A3) are presented in Subsection \ref{subA3}. Finally, (A1), which implies that $X$ must have at least a finite third moment, is needed for two reasons; it guarantees that  $\Exp[\eta^{2}]=\calo(1)$ as $\D \rightarrow \infty$ and at the same time it allows us to apply Theorems \ref{theo:1} and \ref{coro:1} for $r=2$ and obtain the asymptotic behavior of $\Exp[\nu^{2}(t)]$ and $\Exp[N^{2}(t)]$ as $\delta, t \rightarrow \infty$  so that $\delta=o(t)$. These two properties are summarized in the following lemma.

\begin{lemma} \label{lem:330}
(i)If $\Exp[(X^{+})^{3}]<\infty$, then $\Exp[\eta^{2}]=\calo(1)$ as $\D \rightarrow \infty$. 

(ii) If $\Exp[(X^{-})^{3}]<\infty$, then  as $\delta, t \rightarrow \infty$ so that $\delta=o(t)$
\begin{equation} \label{2count}
\; \Exp \Bigl[\Big|\frac{\delta N(t) }{t}-1\Big|^{2}\Bigr] =\calo\Bigl(\frac{\delta}{t}\Bigr)^{2} \quad \text{and} \quad
\Exp \Bigl[\Big|\frac{\delta \nu(t) }{t}-1\Big|^{2}\Bigr] =\calo\Bigl(\frac{\delta}{t}\Bigr)^{2}. 
\end{equation}
\end{lemma}

\begin{proof}
(i) is well-known from renewal theory (see, e.g., Theorem 10.9 in \cite{gut}).  Now, from Theorem 8.1 in \cite{gut} and the fact that $\delta=\Theta(\D)$ it follows that if $\Exp[(X^{-})^{3}]<\infty$, then $\Exp[\tau^{3}]=(\D/\mu)^{3}(1+o(1))=\calo(\delta^{3})$.  Moreover, from (\ref{neon}) we know that $\Var[\tau]=\Theta(\delta)$, therefore, (ii) follows from Theorems \ref{theo:1} and  \ref{coro:1}.
\end{proof}

\begin{proof}[Proof of Theorem \ref{theo:33}] 
(i) From Theorem \ref{ans3}(ii)  it is clear that 
it suffices to show that $\sqrt{t} \, |\tilde{\mu}_{t}- g(\check{\mu}_{t})| \overset{p} \rightarrow 0$  as $\delta, t \rightarrow \infty$  so that $t^{1/4}<< \delta << t$. 

In particular,  since 
\begin{align} \label{w1}
\sqrt{t} \, |\tilde{\mu}_{t}- g(\check{\mu}_{t})| &=  \frac{1}{\sqrt{t}} \, \Big|\sum_{n=1}^{N(t) } \Bigl(\eta_{n} - \rho(\check{\mu}_{t})\Bigr) \Big| \nonumber \\
&\leq  \frac{1}{\sqrt{t}} \Big| \sum_{n=1}^{N(t) } \Bigl(\eta_{n}- \Exp[\eta] \Bigr) \Big| +  \frac{N(t)  \, |\Exp[\eta] - \rho(\mu)|}{\sqrt{t}} +  \frac{N(t)  \,  |\rho(\check{\mu}_{t})- \rho(\mu)|}{\sqrt{t}},
\end{align} 
it  suffices to show that each term in (\ref{w1}) converges to 0 in probability when $t^{1/4}<< \delta << t$. We start with the first term and we observe that, since $\nu(t) =N(t) +1$, using the triangle inequality we can write
$$ \Big| \sum_{n=1}^{N(t) } \Bigl(\eta_{n}- \Exp[\eta] \Bigr) \Big|  \leq \Big| \sum_{n=1}^{\nu(t) } \Bigl(\eta_{n}- \Exp[\eta] \Bigr) \Big|  
+  |\eta_{\nu(t) }- \Exp[\eta]|.$$
Taking expectations and applying the Cauchy-Schwartz inequality to both terms of the right-hand side we obtain 
\begin{align*}
\Exp\Bigl[\Big| \sum_{n=1}^{\nu(t) } \Bigl(\eta_{n}- \Exp[\eta] \Bigr) \Big| \Bigr] &\leq 
\sqrt{ \Var\Bigl[ \sum_{n=1}^{\nu(t) } \Bigl(\eta_{n}- \Exp[\eta] \Bigr) \Bigr]}+ \sqrt{ \Exp\Bigl[ (\eta_{\nu(t) }- \Exp[\eta])^{2} \Bigr]} \\
&\leq \sqrt{ \Var\Bigl[ \sum_{n=1}^{\nu(t) } \Bigl(\eta_{n}- \Exp[\eta] \Bigr) \Bigr]} + \sqrt{ \Exp\Bigl[\sum_{n=1}^{\nu(t) } \Bigl(\eta_{n}- \Exp[\eta] \Bigr)^{2} \Bigr]} \\ 
&= \sqrt{\Exp[\nu(t) ]  \Var[\eta]} + \sqrt{\Exp[\nu(t) ]  \Var[\eta]} ,
\end{align*}
where the second inequality is trivial and the equality follows from an application of Wald's identities.  
But from (\ref{lr}) we have $\Exp[\nu(t)]=\calo(t/\delta)$ as $\delta, t \rightarrow \infty$ so that $\delta=o(t)$, whereas from Lemma \ref{lem:330}(i) we have 
$\Var[\eta]=\calo(1)$ as $\D \rightarrow \infty$. Therefore, 
\begin{align} \label{w2}
 \frac{1}{\sqrt{t}} \; \Exp\Bigl[\Big| \sum_{n=1}^{\nu(t) } \Bigl(\eta_{n}- \Exp[\eta] \Bigr) \Big| \Bigr]  
 =\calo \Bigl(\frac{1}{\sqrt{\delta}}\Bigr).
\end{align}

Regarding the second term in (\ref{w1}), from (A2) and the fact that $\delta =\Theta(\D)$ as $\D \rightarrow \infty$ it follows that $|\Exp[\eta] - \rho(\mu)|= \calo(\delta^{-(r-1)})$ as $\delta \rightarrow \infty$. Moreover, from (\ref{lr}) it follows that $\Exp[N(t)]=\calo(t/\delta)$ as $\delta, t \rightarrow \infty$ so that $\delta=o(t)$. Therefore,
\begin{equation} \label{w4}
\frac{\Exp[N(t) ] \, |\Exp[\eta] - \rho(\mu)|}{\sqrt{t}} = \frac{\calo(t/\delta) \, \calo(\delta^{-(r-1)})}{\sqrt{t}} =  \calo\Bigl( \frac{\sqrt{t}}{ \delta^{r}} \Bigr).
\end{equation}

Regarding the last term in (\ref{w1}), from the  assumption that $\mu \rightarrow \rho(\mu)$ is Lipschitz, there is a constant $C>0$ 
so that $$|\rho(\mu_{2})-\rho(\mu_{1})| \leq C|\mu_{2}-\mu_{1}|, \quad \forall \, \mu_{1}, \mu_{2}>0,$$
and, as a result, 
\begin{align*}
\frac{N(t)  \, |\rho(\check{\mu}_{t})- \rho(\mu)|}{\sqrt{t}}  &\leq C\, \frac{N(t)  \, |\check{\mu}_{t}- \mu|}{\sqrt{t}}  \\
&\leq   C \, \frac{N(t) }{\sqrt{t}} \; \Big|\frac{\delta N(t) }{t} - 1 \Big| + C \, \Exp[\eta] \,  \frac{N^{2}_{t}}{t^{3/2}},
\end{align*}
where the second inequality follows from (\ref{o}). Then, taking expectations and applying the Cauchy-Schwartz inequality in the first term we obtain 
\begin{align} \label{w6}
\Exp \Bigl[\frac{N(t)  \, |\rho(\check{\mu}_{t})- \rho(\mu)|}{\sqrt{t}}  \Bigr] & \leq C \, \sqrt{\frac{\Exp[N^{2}_{t}]}{t} \, \Exp \Big[\Big|\frac{\delta N(t) }{t} - 1 \Big|^{2}\Big]} +  C \, \frac{\Exp[N^{2}_{t}]}{t^{3/2}}  \nonumber \\
&=\calo\Bigl(\frac{1}{\sqrt{t}}\Bigr)+ \calo\Bigl(\frac{\sqrt{t}}{\delta^{2}}\Bigr) = \calo\Bigl(\frac{\sqrt{t}}{\delta^{2}}\Bigr).
\end{align}
The second  equality follows from (\ref{2count}), which implies that $\Exp[N^{2}_{t}]=\calo(t^{2}/\delta^{2})$.

From (\ref{w1}) - (\ref{w6}) we obtain 
\begin{equation} \label{w8}
\sqrt{t} \, |\tilde{\mu}_{t}- g(\check{\mu}_{t})| 
= \calo \Bigl(\frac{1}{\sqrt{\delta}}\Bigr)+  \calo\Bigl(\frac{\sqrt{t}}{\delta^{r}}\Bigr) + \calo\Bigl(\frac{\sqrt{t}}{\delta^{2}}\Bigr) ,
\end{equation}
which completes the proof, since the last term in the right hand side is the dominant one.

(ii) Due to Theorem \ref{ans3}(i), 
it   suffices to show that $\sqrt{t} |\bar{\mu}_{t}-g(\hat{\mu}_{t})| \overset{p} \rightarrow 0$ as $\delta, t \rightarrow \infty$  so that $ t^{1/4} << \delta <<t$. Since 
\begin{align} \label{w2222}
\sqrt{t} |\bar{\mu}_{t}-g(\hat{\mu}_{t})| &= \frac{\sqrt{t}}{\tau(t) } \, \Big|\sum_{n=1}^{N(t) } \Bigl(\eta_{n} - \rho(\hat{\mu}_{t})\Bigr) \Big|
= \frac{t}{\tau(t) } \, \frac{1}{\sqrt{t}} \Big|\sum_{n=1}^{N(t) } \Bigl(\eta_{n} - \rho(\hat{\mu}_{t})\Bigr) \Big| \nonumber \\
&\leq \frac{t}{\tau(t) } \,  \Bigl[ \frac{1}{\sqrt{t}} \Big| \sum_{n=1}^{N(t) } \Bigl(\eta_{n}- \rho(\check{\mu}_{t}) \Bigr) \Big| 
+ \frac{N(t)  \,  |\rho(\hat{\mu}_{t})- \rho(\check{\mu}_{t})|}{\sqrt{t}} \Bigr],
\end{align} 
from (\ref{taut}) and (\ref{w8}) it is clear that it suffices to show that the second term in the  parenthesis in  (\ref{w2222}) converges to 0 in probability  as $\delta, t \rightarrow \infty$ so that $\delta=o(t)$. Indeed, due to the assumption that $\mu \rightarrow \rho(\mu)$ is Lipschitz, we have  
\begin{align*}
\frac{N(t)  \,  |\rho(\hat{\mu}_{t})- \rho(\check{\mu}_{t})|}{\sqrt{t}} &\leq C\,  \frac{N(t) }{\sqrt{t}} \, \Big| \frac{\D \, N(t) }{\tau(t)}-\frac{\D \, N(t) }{t} \Big| \\
&= C \frac{\D}{\delta} \, \frac{t}{\tau(t) } \, \Bigl(\frac{\delta N(t) }{t}\Bigr)^{2} \; \frac{t-\tau(t)}{\delta \sqrt{t}}.
\end{align*}
From (\ref{2count}) and (\ref{age}) it is clear that the right-hand side converges to 0 in probability  as $\delta, t \rightarrow \infty$ so that $\delta=o(t)$,  which completes the proof. 
\end{proof}

\subsubsection{Sufficient conditions for (A2)} \label{subA2}
In order to find sufficient conditions for (A2),  we  appeal to Stone's  refinements of the renewal theorem \cite{sto1}, \cite{sto2}.
Thus, let  $U(\D)$ be the renewal function and $(\tau_{n,+})_{n \in \mathbb{N}}$ the ascending ladder times of the random walk $(S_{t})_{t \in \mathbb{N}}$, i.e., $U(\D):= \sum_{n=0}^{\infty} \Pro(S_{\tau_{n,+}} \leq \D)$ and
$$\tau_{n,+} := \inf\{t \geq \tau_{n-1,+}: S_{t} > S_{\tau_{n-1,+}}\} ; \; \tau_{0,+}:=0.$$
From \cite{sto1} we know that if $X$ is strongly non-lattice, i.e., 
$$\liminf_{|\theta| \rightarrow \infty}|1-\Exp[e^{i \theta X}]|>0$$
and (A1) holds,  i.e.,  $\Exp[|X|^{2+r}]<\infty$  for some $r  \geq 2$,
then 
\begin{equation} \label{qq2}
\Big|  U(\D)- \frac{\D}{\mu} - \frac{\rho(\mu)}{\mu} \Big|= o\Bigl(\frac{1}{\D^{r-1}} \Bigr).
\end{equation}
Moreover, from \cite{sto2} it follows that if $X$ is strongly non-lattice and it has an exponentially decaying right tail,
in the sense that $\Pro(X>x)=o(e^{-\theta_{1} x})$ as $x \rightarrow \infty$ for some $\theta_{1}>0$, then 
\begin{equation} \label{qq3}
 \Big|  U(\D)- \frac{\D}{\mu} - \frac{\rho(\mu)}{ \mu} \Big|= o(e^{-\theta_{2} \D} )  \quad \text{as} \quad  \D \rightarrow \infty
\end{equation} 
for some $\theta_{2}>0$. Based on these results, we have the following lemma that provides sufficient conditions for (A2).

\begin{lemma} \label{lem:33}
(A2) is satisfied when $X$ is strongly non-lattice and one of the following holds
\begin{enumerate}
\item[(i)]  $\Pro(X>0)=1$ and $\Exp[|X|^{2+r}]<\infty$  for some $r  \geq 2$,
\item[(ii)] $\Pro(X>x)=o(e^{-\theta_{1} x})$ as $x \rightarrow \infty$ for some $\theta_{1}>0$.
\end{enumerate}
\end{lemma}

\begin{proof}[Proof of Lemma \ref{lem:33}]
We observe that 
\begin{align} \label{qq1}
|\Exp[\eta]-\rho(\mu)| &= |\Exp[S_{\tau}-\D] -\rho(\mu)| =  \mu \, \Big|  \Exp[\tau]- \frac{\D}{\mu} - \frac{\rho(\mu)}{ \mu} \Big|,
\end{align} 
where the second equality follows from an application of Wald's identity. When $\Pro(X>0)=1$, then  $\Exp[\tau]= U(\D)$ and the Lemma 
follows from (\ref{qq2}). When $X$ has an exponentially decaying right tail, using (\ref{qq3}) and the following representation of the expected overshoot  
$$\Exp[\eta] = \int_{0}^{\infty} \int_{0}^{\infty} \Bigl(1-\Pro(S_{\tau+} >\D+y-x) \Bigr) \, U(dx) \, dy,$$ 
we can show, working similarly to Chang \cite{cha}, p. 723, that $|\Exp[\eta]-\rho(\mu)| \rightarrow 0$ \textit{exponentially fast} as $\D \rightarrow \infty$, which of course implies (A2).
\end{proof}

\subsubsection{Sufficient conditions for (A3)} \label{subA3}
 When $\Pro(X>0)=1$, from (\ref{asyover}) it follows that
$$\rho(\mu)=  \frac{\Exp[X^{2}]}{2 \Exp[X]}=\frac{\mu}{2} + \frac{\sigma^2}{2\mu}.$$
Therefore, for (A3) to hold, $\sigma^2/\mu$ has to be Lipschitz as a function of $\mu$. Clearly, this is not the case when $\sigma$ is independent of $\mu$, unless $\mu$ is restricted on a compact interval. However,  when $\sigma^{2}=c \mu^{2}$ for some $c>0$, then 
$\rho(\mu)= [(1+c)/2] \, \mu$ and (A3) is satisfied. For example, if $X$ follows the Gamma distribution with shape parameter $k>0$ and rate parameter $\lambda>0$, 
then $\sigma^{2}=\mu^{2}/k$ and $\rho(\mu)$ is proportional to $\mu$ for any given  $k$.  

When $\Pro(X>0)<1$, $\rho(\mu)$ does not typically admit a convenient closed-form expression in terms of $\mu$. An exception is the Gaussian distribution, in which case (\ref{asyover}) takes the following form
$$\rho(\mu)= \frac{\mu^{2}+\sigma^{2}}{2 \mu} 
-\sigma \; \sum_{n=1}^{\infty} \frac{\phi(b_{n})  -   b_{n} \, \Phi(-b_{n})}{\sqrt{n}} \, , 
\quad b_{n}:= \frac{\mu \, \sqrt{n}}{\sigma},$$
(see \cite{wood}, p. 34), where $\Phi(\cdot)$ and $\phi(\cdot)$ are the c.d.f. and p.d.f. respectively of the standard normal distribution. In this case as well, $\sigma^{2}=c \mu^{2}$ for some $c>0$ implies  $b_{n}=\sqrt{n/c}$ and, consequently, $\rho(\mu)=w_{c} \, \mu$, where 
$$w_{c}:= \frac{1+c}{2}-  \; \sum_{n=1}^{\infty} \frac{\phi(\sqrt{n/c})  -    \sqrt{n/c}\, \Phi(-\sqrt{n/c})}{\sqrt{n/c}}.$$


\subsection{Summary and simulation experiments} \label{sub34}
In Table \ref{tab:1}, we present the estimators of $\mu$ that we have considered so far in the context of the sampling scheme (\ref{modelo}).
For each estimator, we report whether it requires knowledge of the overshoots $(\eta_{n})$ and/or the distribution of $X$ and we 
present the sampling rates for which its performance is optimized. 

In Figure \ref{fig:1} we  plot the relative efficiency of each estimator as a function of the sampling period, $\delta$, for a fixed horizon of observations, $t=300$, when  $X\sim \cN(\mu, c\mu^{2})$ and $c=\mu=4$. More specifically, we define the relative efficiency  of $\bar{\mu}_{t}$ (similarly for the other estimators) as the ratio of its mean square error over the mean square error of $S_{t}/t$,
$$RE(\bar{\mu}_{t}):= \frac{\Exp[(\bar{\mu}_{t}-\mu)^{2}]}{\Exp[(\mu_{t}-\mu)^{2}]}=\frac{\Exp[(\bar{\mu}_{t}-\mu)^{2}]}{\sigma^{2}/t}$$
and we  compute it using simulation experiments.  


The findings depicted in this figure verify our asymptotic results. First of all, we observe that $\bar{\mu}_{t}$, $\hat{\mu}_{t}$ and $g(\hat{\mu}_{t})$ are more efficient than $\tilde{\mu}_{t}$, $\check{\mu}_{t}$ and $g(\check{\mu}_{t})$, respectively,   for \textit{any} $\delta$ and especially for large values of $\delta$. Thus, it is always preferable to use $\tau(t)$ in the denominator, instead of $t$, especially when one is interested in large sampling periods.

Moreover, we see that for $\bar{\mu}_{t}$ and $\tilde{\mu}_{t}$, the smaller the sampling period $\delta$, the better the performance; for 
$\hat{\mu}_{t}$ and $g(\hat{\mu}_{t})$,  performance improves (and eventually remains flat) as $\delta$ increases;
for $\check{\mu}_{t}$ and $g(\check{\mu}_{t})$, performance  is optimized when $\delta$ is in a particular range and deteriorates for both very small and very large values of  $\delta$. Therefore,  when the random walk $(S_{t})$ is fully observed at the sampling times $(\tau_{n})$, it is preferable to have a high sampling rate, but when the overshoots $(\eta_{n})$ are unobserved and $\tau(t)$ (resp. $t$) is used in the denominator, the sampling rate should be low (resp. moderate). 
 
Finally, we  observe  that $g(\check{\mu}_{t})$ is more efficient than $\check{\mu}_{t}$ \textit{uniformly over $\delta$}, whereas its performance attains its optimum over a much wider range for $\delta$. On the other hand, $g(\hat{\mu}_{t})$ is more efficient than $\hat{\mu}_{t}$  only for small $\delta$. For very large values of $\delta$, $\hat{\mu}_{t}$ turns out to be more efficient even than $\bar{\mu}_{t}$! Therefore,  when using  $\tau(t)$ (resp. $t$) in the denominator, approximating the overshoots is beneficial only for high sampling rates (resp. for any sampling rate and especially for low sampling rates).



\begin{figure}[!h]
  \centering
  \includegraphics [width=0.9\linewidth, height=0.6\linewidth]{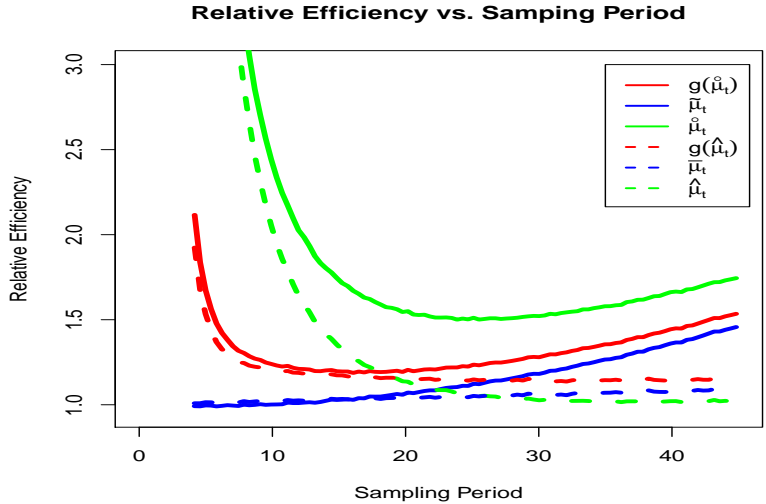}
\label{fig:1}
\caption{Relative efficiency against sampling period for each of the estimators in Table \ref{tab:1} in a curved Gaussian distribution. The horizon of observations is $t=300$, the true mean is $\mu=4$ and the variance $\sigma^{2}=4 \mu^{2}$. The computations are based on simulation experiments.}
\end{figure}

\begin{table}[!h]
	\centering
			\begin{tabular}{|c|c|c|c|c|} \hline
	   Estimator          & Formula & Distribution & Overshoots   & Optimal Rate            \\ \hline  \hline
	 $\bar{\mu}_{t}$  &  $S_{\tau(t)}/ \tau(t)$ & no & yes   &  $\delta <<t$   \\ 
	 \hline 
	 $\tilde{\mu}_{t}$    &  $S_{\tau(t)}/t$ & no &  yes  &  $\delta <<\sqrt{t}$  \\ 
	 \hline
	 $\hat{\mu}_{t}$      &  $N(t)  \D / \tau(t)$ & no & no &  $\sqrt{t} << \delta << t$   \\ 
	 \hline
	 $\check{\mu}_{t}$    &  $ N(t)  \D / t$ & no  & no  &   $\delta \sim \sqrt{t}$            \\ 
	 \hline   
	 $g(\hat{\mu}_{t})$   &  $g(\hat{\mu}_{t})   $ & yes & no  & $t^{1/4} << \delta << t$  \\ 
	 \hline 
	 $g(\check{\mu}_{t})$   &  $g(\check{\mu}_{t})$  & yes & no   &  $ t^{1/4} << \delta << \sqrt{t}$   \\ 
	 \hline 
	 \end{tabular}
	 	\caption{Estimators of $\mu$ based on the sampling times (\ref{modelo}) (the first two apply to more general sampling schemes). 
	 	For each estimator, we report whether it requires knowledge of 	the random walk distribution,	the overshoots $(\eta_{n})$, and we specify the range of of sampling rates for which its performance is optimized.}
	 		\label{tab:1}
	\end{table}


\subsection{Estimating the standard deviation} \label{sub35}
In order to attach an (asymptotic) standard error to the estimators we considered in this section when $\sigma$ is unknown, 
we need  a consistent estimator of $\sigma$, which is not possible to obtain using only the sampling times $(\tau_{n})$. 
If, however, in addition to $(\tau_{n})$ we also observe the following stopping times:
$$
\theta_{n}:= \inf\{t \geq \theta_{n-1}: Z_{t}- Z_{\theta_{n-1}} \geq \G \}, \quad n \in \mathbb{N},
$$
where $Z_{t}:= (X_{1})^{2}+ \ldots+  (X_{t})^{2}$ and  $\G>0$, then we can estimate $\sigma$ at some arbitrary time $t\geq \theta_{1}$ with 
$$\hat{\sigma}_{t}:= \sqrt{\frac{\G M_{t}}{\theta(t)} - (\hat{\mu}_{t})^{2}},$$
where $M(t):=\max\{n: \theta_{n} \leq t\}$ and $\theta(t):=\theta_{M(t)}$. The following theorem shows that $\hat{\sigma}_{t}$ is a consistent estimator of $\sigma$ under low-rate sampling and, consequently, it implies that  Theorems \ref{theo:32}(iii) and \ref{theo:33} remain valid if we replace $\sigma$ by $\hat{\sigma}_{t}$.

\begin{theorem} \label{last}
If $\Exp[|X|^{3}]<\infty$, then  $\hat{\sigma}_{t} \overset{p} \rightarrow \sigma$  as $\D, \G, t \rightarrow \infty$ so that $\D, \G=o(t)$.
\end{theorem}

\begin{proof}
Due to Theorem \ref{theo:32}(ii), it suffices to prove  that 
$$\frac{\G M(t)}{\theta(t)}  \overset{p} \longrightarrow \Exp[X^{2}]=\sigma^{2}+\mu^{2}$$
 as $\G, t \rightarrow \infty$ so that $\G=o(t)$. Under a third moment assumption, this can be done in exactly the same way as we showed that $\hat{\mu}_{t} \overset{p} \rightarrow \mu$ when $\D, t \rightarrow \infty$ so that $\D=o(t)$ in Theorem \ref{theo:32}(ii). 
\end{proof}

\section{Decentralized parameter estimation} \label{sec4}
We now apply the results of Section \ref{sec3} to  the decentralized parameter estimation that we described in Subsection \ref{sub13} of the Introduction. Thus, we consider $K$ sensors, dispersed at various locations, so that the observations at each sensor $k$, $(X_{t}^{k})_{t \in \mathbb{N}}$, are i.i.d. with unknown mean $\mu$ and standard deviation $\sigma_{k}$,  $k=1, \ldots, K$.  All sensors communicate with a fusion center, whose goal is to estimate $\mu$. When $\sigma_{1}, \ldots, \sigma_{k}$ are known and each sensor transmits its exact observation $X_{t}^{k}$ at every time $t$, the fusion center can use the best linear estimator of $\mu$ at any time $t$, which is given by (\ref{opti}), and whose asymptotic distribution as  $t \rightarrow \infty$ is given by (\ref{CLT}), under the assumption of independence across sensors. However, due to bandwidth and energy constraints, the sensors are not, typically, able to transmit their complete observations to the fusion center. Instead, they should ideally transmit, infrequently, low-bit messages. 
Our goal in this section is to show that it is possible to construct estimators that preserve asymptotic distribution (\ref{CLT}), even under such severe communication constraints. Indeed, assuming for simplicity that $\mu>0$, each sensor $k$ needs to communicate with  the fusion center at the following sequence of stopping times
\begin{equation} \label{times}
\tau_{n}^{k}:= \inf\{t \geq \tau_{n-1}: S^{k}_{t}- S^{k}_{\tau^{k}_{n-1}} \geq \D^{k} \}, \quad n \in \mathbb{N};  \quad \tau_{0}^{k}:=0 ,
\end{equation} 
where $\D^{k}>0$ is a fixed threshold. Then, two natural estimators of $\mu$ at some time $t$ are given by 
\begin{align}
\check{\mu}_{t} &:=\sum_{k=1}^{K} w_{k} \, \check{\mu}_{t}^{k}, \quad \check{\mu}_{t}^{k}:= \Di N_{t}^{k}/ t  \label{old2} \\
\hat{\mu}_{t} &:=\sum_{k=1}^{K} w_{k} \, \hat{\mu}_{t}^{k}, \quad \hat{\mu}_{t}^{k}:= \Di N_{t}^{k}/\tau^{k}(t), \label{new2} 
\end{align}
where $N_{t}^{k}=\max\{ n: \tau_{n} <t\}$ is the number of messages transmitted by sensor $k$ up to time $t$ and $\tau^{k}(t)=\tau_{N^{k}(t)}$ the last communication time from sensor $k$. If, additionally, the form of the limiting average overshoot 
$$\rho_{k}(\mu):=\lim_{\Di \rightarrow \infty} \Exp[S_{\tau^{k}}-\Di] ; \quad  \tau^{k}:=\tau_{1}^{k} $$ 
is known for each $k$, two alternative estimators of $\mu$ are $\sum_{k=1}^{K} w_{k} \, g_{k}(\check{\mu}_{t}^{k})$ and $\sum_{k=1}^{K} w_{k} \, g_{k}(\hat{\mu}_{t}^{k})$, where  $g_{k}(x):= x(1+ \rho_{k}(x)/\Di)$. The following theorem describes the asymptotic behavior of these estimators.
Its proof is a direct consequence of the results presented in Section \ref{sec3} and the assumption of independence across sensors.  In order to state it, we set 
$\uD:= \min_{1 \leq k \leq K}\D^{k}$ and $\oD:=\max_{1 \leq k \leq K} \Di$.

\begin{theorem} \label{theo:3222}
\begin{enumerate}
\item[(i)] 
If $\uD,t \rightarrow \infty$ so that $\sqrt{t} << \uD$ and $\oD << t$, then
$$
\frac{\hat{\mu}_{t}-\mu}{\sqrt{K} \sigma /\sqrt{t}} \rightarrow \cN(0,1).
$$
\item[(ii)] Suppose that the assumptions of Theorem \ref{theo:33} are satisfied by each $X_{1}^{k}$, $\rho_{k}(\cdot)$. 
If $\uD,t \rightarrow \infty$ so that $\sqrt[4]{t} << \uD$ and $\oD<< t$, then
$$
\frac{\sum_{k=1}^{K} w_{k} \, g_{k}(\hat{\mu}_{t}^{k})-\mu}{\sqrt{K} \sigma /\sqrt{t}} \rightarrow \cN(0,1).
$$
If, additionally, $\oD << \sqrt{t}$, then
$$
\frac{\sum_{k=1}^{K} w_{k} \, g_{k}(\check{\mu}_{t}^{k})-\mu}{\sqrt{K} \sigma /\sqrt{t}} \rightarrow \cN(0,1).
$$\end{enumerate}

\end{theorem}


When $\sigma_{k}=\sigma$ for every $1 \leq k \leq K$, then $w_{k}=1/K$ and Theorem \ref{theo:3222} will remain valid if we replace $\sigma$ by 
a consistent estimator of it. Following Subsection \ref{sub35}, we can find such an estimator based on infrequent transmissions of 1-bit messages from all sensors. To this end, each sensor $k$ needs to communicate with the fusion center, in addition to (\ref{times}), at the following sequence of stopping times: 
$$
\theta_{n}^{k}:= \inf\{t \geq \theta_{n-1}^{k}: Z^{k}_{t}- Z^{k}_{\theta^{k}_{n-1}} \geq \G^{k} \}, \quad n \in \mathbb{N},
$$
where $Z_{t}^{k}:= (X_{1}^{k})^{2}+ \ldots+  (X_{t}^{k})^{2}$ and $\G^{k}>0$.  Then, 
$$\hat{\sigma}_{t}:= \sqrt{ \sum_{k=1}^{K} w_{k} \, \frac{\G^{k} M_{t}^{k}}{\theta^{k}(t)} - (\hat{\mu}_{t})^{2}}$$
turns out to be a consistent estimator of $\sigma$,  where $M^{k}(t):=\max\{n: \theta_{n}^{k} \leq t\}$ and $\theta^{k}(t):=\theta^{k}_{M^{k}(t)}$. This is the content of the following theorem, for which we set 
$\uG:= \min_{1 \leq k \leq K} \Gi$ and $\oG:=\max_{1 \leq k \leq K} \Gi$.

\begin{theorem}
If $\Exp[|X^{k}_{1}|^{3}]<\infty$ and $\sigma_{k}^{2}=\sigma^{2}$ for every $1 \leq k \leq K$, then 
$\hat{\sigma}_{t} \overset{p}  \rightarrow \sigma$ as $\uG,\uD, t \rightarrow \infty$ so that $\oG, \oD=o(t)$.
\end{theorem}

\begin{proof}
It suffices to show that   for every $1 \leq k \leq K$,
$$ \frac{\G^{k} M_{t}^{k}}{\theta^{k}(t)} \overset{p} \rightarrow \Exp[(X_{1}^{k})^{2}]= \sigma^{2}+\mu^{2}$$
as $\Gi,t \rightarrow \infty$ so that $\Gi=o(t)$, which can be done  as in Theorem \ref{last}.
\end{proof}

Finally, when the sign of $\mu$ is not known in advance, we need to replace the communication times (\ref{times}) with $\tau_{n}^{k}:= \inf\{t \geq \tau_{n-1}^{k}: |S^{k}_{t}- S^{k}_{\tau^{k}_{n-1}}| \geq \D^{k} \}$ and at each time $\tau_{n}^{k}$ sensor $k$  needs to transmit the 1-bit message  $z_{n}^{k}=1(S^{k}_{\tau_{n}^{k}}-S^{k}_{\tau_{n-1}} \geq  \Di)$. Then, the estimators (\ref{old2}) and (\ref{new2}) are generalized into  (\ref{old}) and (\ref{new}), respectively, and Theorem \ref{theo:3222}(i) remains valid, however, the extension of Theorem \ref{theo:3222}(ii) is not straightforward.

\section{Conclusions} \label{sec5}
In the present work, we extended some fundamental renewal-theoretic results for renewal processes whose rates go to 0 and for random walks whose drifts go to infinity.  We applied these extensions to a problem of parameter estimation subject to communication constraints, but we believe that they can be useful 
in a variety of setups where recurrent events occur infrequently. It remains an open problem to examine under what conditions, if any, other classical results from renewal theory remain valid in such a low-rate setup.


\addcontentsline{toc}{chapter}{Bibliography}

\end{document}